\documentclass[12pt,a4paper,twoside,reqno]{amsart}
\allowdisplaybreaks

\usepackage[all,cmtip]{xy}
\usepackage{amsmath,amscd}
\usepackage{amssymb,amsfonts}
\usepackage{graphicx}
\usepackage[driver=pdftex,margin=3cm,heightrounded=true,centering]{geometry}
\usepackage{mathtools}
\usepackage{tensor}
\usepackage{url}
\usepackage[colorlinks=true,linkcolor=blue,citecolor=blue]{hyperref}
\usepackage{slashed}
\usepackage[new]{old-arrows}
\usepackage{todonotes}

\usepackage{graphicx}

\usepackage{a4wide}

\usepackage{thmtools}
\declaretheorem[style=theorem,name={Theorem}]{theoremletter}

\newtheorem{introcorollary}[theoremletter]{Corollary}

\usepackage{amsthm}

\theoremstyle{plain}

\newtheorem{theorem}{Theorem}[section]
\newtheorem*{thm*}{Theorem}
\newtheorem{prop}[theorem]{Proposition}
\newtheorem{lemma}[theorem]{Lemma}
\newtheorem{cor}[theorem]{Corollary}

\theoremstyle{definition}
\newtheorem{dfn}[theorem]{Definition}

\theoremstyle{remark} 

\newtheorem{example}[theorem]{Example}

\newtheorem{remark}[theorem]{Remark}

\theoremstyle{plain}

\usepackage{amscd}

\numberwithin{equation}{section}

\newcommand{\alpheqn}[1][\relax]{
     \refstepcounter{equation}
     \if#1\relax \relax
       \else \label{#1}
     \fi  
     \setcounter{saveeqn}{\value{equation}}%
    \setcounter{equation}{0}%
    \renewcommand{\theequation}{\thealphequation}}
\newcommand{\reseteqn}{\setcounter{equation}{\value{saveeqn}}%
     \renewcommand{\theequation}{\thearabicequation}}

\providecommand{\mathscr}{\mathcal} 
\IfFileExists{mathrsfs.sty}{
\usepackage{mathrsfs}}         
   {
     \IfFileExists{eucal.sty}{
        \usepackage[mathscr]{eucal}} 
   {
   }
}

\newcommand{\tens}{\otimes}

\newcommand{\varps}{{\varepsilon}}
\newcommand{\vertiii}[1]{{\left\vert\kern-0.25ex\left\vert\kern-0.25ex\left\vert #1 
    \right\vert\kern-0.25ex\right\vert\kern-0.25ex\right\vert}}
\newcommand{\Bvert}[1]{{\Big\vert\kern-0.25ex\Big\vert\kern-0.25ex\Big\vert #1 
    \Big\vert\kern-0.25ex\Big\vert\kern-0.25ex\Big\vert}}
\newcommand{\bvert}[1]{{\big\vert\kern-0.25ex\big\vert\kern-0.25ex\big\vert #1 
    \big\vert\kern-0.25ex\big\vert\kern-0.25ex\big\vert}}
\newcommand{\nvert}[1]{{\vert\kern-0.25ex\vert\kern-0.25ex\vert #1 
    \vert\kern-0.25ex\vert\kern-0.25ex\vert}}

\renewcommand{\leq}{\leqslant}
\renewcommand{\geq}{\geqslant}

\newcommand{\op}{\operatorname{op}} 

\newcommand{\nn}{\mathbb{N}}

\newcommand{\cc}{\mathbb{C}}

\newcommand{\C}[1]{\mathcal{#1}}

\newcommand{\T}[1]{\textup{#1}}

\newcommand{\inn}[1]{\left\langle #1 \right\rangle}

\newcommand{\red}{{\operatorname{red}}}

\newcommand{\A}{{\mathcal{A}}}
\newcommand{\BB}{{\mathbb{B}}}
\newcommand{\CC}{{\mathbb{C}}}
\newcommand{\NN}{{\mathbb{N}}}
\renewcommand{\S}{{\mathcal{S}}}

\newcommand{\GG}{{\mathbb{G}}}
\newcommand{\Pol}{\operatorname{Pol}}
\newcommand{\Irred}{\operatorname{Irred}}
\newcommand{\id}{{\operatorname{id}}}
\newcommand{\X}{{\mathcal{X}}}

\newcommand{\bbGamma}{{\mathpalette\makebbGamma\relax}}
\newcommand{\makebbGamma}[2]{%
  \raisebox{\depth}{\scalebox{1}[-1]{$\mathsurround=0pt#1\mathbb{L}$}}%
}

\makeatletter
\@namedef{subjclassname@2020}{%
  \textup{2020} Mathematics Subject Classification}
\makeatother

\begin{document}

\author{Are Austad}
\address{Are Austad, Department of Mathematics, University of Oslo, P.O. Box 1053 Blindern, N-0316 Oslo, Norway}
\email{areaus@math.uio.no}

\author{David Kyed}
\address{David Kyed, Department of Mathematics and Computer Science, University of Southern Denmark, Campusvej 55, DK-5230 Odense M, Denmark}
\email{dkyed@imada.sdu.dk}



\title{Quantum metrics from length functions on quantum groups}

\begin{abstract}
We study the quantum metric structure arising from length functions on quantum groups and show that for coamenable quantum groups of Kac type, the quantum metric information is captured by the algebra of central functions. Using this, we provide the first examples of length functions on (genuine) quantum groups which give rise to compact quantum metric spaces.

\end{abstract}

\maketitle

\section{Introduction}
Building on the classical Gelfand correspondence between compact Hausdorff spaces and commutative unital $C^*$-algebras, numerous non-commutative analogues of classical point set based theories have emerged over the past 50 years. Prominent examples include the theory of quantum groups \cite{KlSc:QGR}, Connes' noncommutative geometry \cite{Con:NCG} and Rieffel's theory of compact quantum metric spaces \cite{Rie:CQM}, the latter of which is the primary object of study in the present paper. As the name suggests, Rieffel's theory is a non-commutative extension of the theory of compact metric spaces, and has provided a mathematically rigorous framework within which heuristic statements from physics can be proven rigorously. A  prime example of this phenomenon is Rieffel's seminal result that matrix algebras (fuzzy spheres) converge to the 2-sphere \cite{Rie:MSG}.  \\

More formally, a compact quantum metric space consists of an operator system $\X$ equipped with a seminorm $L\colon \X \to [0,\infty)$ such that the associated  \emph{Monge-Kantorovi\v{c} metric}
\[
d_L(\mu,\nu):=\sup\big\{|\mu(x)-\nu(x)| \mid L(x)\leq 1\big\} \qquad (\mu,\nu \in \S(\X))
\]
metrizes 
the weak$^*$ topology on the state space $\S(\X)$; see Section \ref{sec:cqms} for the precise requirements and examples. \\
The theory of compact quantum metric spaces is closely linked to  Connes' non-commutative (differential) geometry by means of the following construction: Given a spectral triple $(\A,H,D)$, one obtains a seminorm $L_D$ on $\A$ by setting $L_D(a):=\| [D,a]\|$, and this construction often (but not always) yields interesting examples of compact quantum metric spaces \cite{OzawaRieffel2005, ChristRieffel2017, AgKa:PSM, Christensen-Ivan:spectral-triples}. It is generally a  non-trivial task to verify if a given spectral triple gives rise to a compact quantum metric space, and even some of the most fundamental examples in the theory are not fully understood. As an example, if $\Gamma$ is a discrete group equipped with a proper length function $\ell\colon \Gamma \to [0,\infty)$, one obtains a natural spectral triple $(\CC\Gamma, \ell^2(\Gamma), D_\ell)$ (see Section \ref{sec:cqms} for details) and the only classes of groups for which it is known that $L_\ell:=L_{D_\ell}$ provides a compact quantum metric structure, are hyperbolic groups \cite{OzawaRieffel2005} and groups  of polynomial growth \cite{ChristRieffel2017}. In both cases, the argument is rather involved, and except for minor improvements \cite{LongWu:twisted}, no new examples have been obtained since Christ and Rieffel's paper \cite{ChristRieffel2017} from 2017.  {It therefore seems natural to attempt to expand the number of examples by allowing for more flexibility in the defining data, as was done for instance in \cite{GerontogiannisMesland2025} and \cite{Antonescu-Christensen}, and the present paper suggests a new direction to this strategy by extending the domain from groups to quantum groups. }\\

The structure of a compact quantum group is encoded in a unital $C^*$-algebra $A$ endowed with a $*$-homomorphism $\Delta\colon A\to A\tens A$ satisfying certain axioms (spelled out in Section \ref{sec:qgrps}) and one may view $A$ in two ways: Firstly, $A$ can be thought of as a quantum analogue of the continuous functions on a compact topological group, in which case it is often denoted $C(\GG)$ to reflect this point of view. Secondly, one may view $A$ as a generalisation of a reduced group $C^*$-algebra, in which case it is more naturally denoted $C^*_{\red}( \bbGamma)$ and one thinks of $ \bbGamma$ as the discrete dual quantum group of $\GG$. Although we will be generalising constructions from the class of group $C^*$-algebras, we have opted for the $C(\GG)$-notation to conform with most of the literature upon which we base our analysis. In addition to the $C^*$-algebraic picture, a compact quantum group also gives rise to a Hopf $*$-algebra $\Pol(\GG)$ as well as a von Neumann algebraic quantum group $L^\infty(\GG)$, which both capture all the structure present.   In the quantum group setup, a length function is defined as a function $\ell$ on the set $\Irred(\GG)$ of equivalence classes of irreducible corepresentations, satisfying a natural set of axioms (see Definition \ref{def:length-function}). From such a length function, one obtains a spectral triple in  the same manner as for discrete groups, but to the best of our knowledge, there are no known genuine quantum examples (i.e.~beyond finite-dimensional  quantum groups and reduced $C^*$-algebras of {certain} classical discrete groups)  where the associated commutator-seminorm {$L_\ell$} yields a compact quantum metric space.  The aim of the present paper is to provide a criterion ensuring that this is the case and to utilise it to construct the first natural quantum group examples. 
More precisely, we show that under natural, albeit somewhat restrictive, assumptions the algebra of central functions
\[
C_z(\GG):=\{a\in C(\GG) \mid \sigma \Delta(a)=\Delta(a)\}
\]
is able to detect when $(C(\GG),L_\ell)$ is a compact quantum metric space. Here $\sigma$ denotes the flip automorphism on $C(\GG) \tens C(\GG)$.

\begin{theoremletter}\label{introthm:lifting}
If $\GG$ is a compact, coamenable quantum group of Kac type and $\ell\colon \Irred(\GG) \to [0,\infty)$ is a proper length function, then $(C(\GG), L_\ell)$ is a compact quantum metric space if and only if $(C_z(\GG), L_\ell)$ is a compact quantum metric space.
\end{theoremletter}
Here the notion of coamenability is the quantum analogue of amenability of a discrete group, while the Kac type assumption can be viewed as unimodularity of the discrete dual quantum group  --- an assumption that is of course vacuous in the setting of discrete groups where unimodularity is automatic; see Section \ref{sec:qgrps} for more details. \\
The proof of Theorem \ref{introthm:lifting} builds on the approximation techniques developed in \cite{AKK:Podcon}, the main novelty being the utilization of a sequence of \emph{central} functionals approximating the counit. The existence of such a sequence is ensured by the coamenability assumption, while the Kac type assumption provides a conditional expectation $E\colon C(\GG) \to C_z(\GG)$  which, in turn, ensures that the Monge-Kantorovi\v{c} distance between two central states can be detected by their restriction to the algebra of central functions (see Lemma \ref{lem:restriction}). \\
 It is well known that $C_z(\GG)$ and its Hopf-algebraic counterpart, $\Pol_z(\GG)$, are generally more robust  than the quantum group itself. As an example, for Woronowicz'  famous deformation $SU_q(2)$ of the Lie group $SU(2)$, the algebra of central functions $C_z(SU_q(2))$ is independent of $q$  and, in particular, isomorphic to the classical algebra of class functions $C_z(SU(2))$ on $SU(2)$. 
 The polynomial algebra of central functions $\Pol_z(\GG)$   is $*$-isomorphic to the so-called fusion algebra $F(\GG)$, which is constructed as the free vector space with basis $\Irred(\GG)$ and algebra structure induced by direct sum and tensor product of corepresentations and involution induced by conjugation. A length function $\ell$ on $\GG$ also gives rise to a spectral triple based on $F(\GG)$, and we show  in Corollary \ref{cor:fusion-vs-central} that if the associated commutator seminorm 
  yields a compact quantum metric structure on $F(\GG)$, then so does the restriction of the original seminorm $L_\ell$ to $C_z(\GG)$. From this and Theorem \ref{introthm:lifting}  we get the following stability result:

\begin{introcorollary}\label{cor:equivalent-fusion}
Let  $\GG_1$ and $\GG_2$ be compact quantum groups with length functions $\ell_1$ and $\ell_2$ and assume that  $\ell_1$ defines a compact quantum metric structure on $F(\GG_1)$. If there exists a bijection $\alpha\colon \Irred(\GG_1) \to \Irred(\GG_2)$ which intertwines the length functions and extends to a $*$-isomorphism $ F(\GG_1) \simeq F(\GG_2)$, and if $\GG_2$ is coamenable and of Kac type,  then  $\ell_2$ yields a compact quantum metric structure on both $F(\GG_2)$ and $C(\GG_2)$.
\end{introcorollary}

The quantum counterpart to a classical compact Lie group is known as a \emph{compact matrix quantum group}, and the additional data is encoded in a fixed  (not necessarily irreducible) corepresentation $u$ whose matrix coefficients generate $\Pol(\GG)$ as a $*$-algebra. One refers to the chosen $u$ as the \emph{fundamental corepresentation}. 
Natural examples of matrix quantum groups include $SU_q(2)$ as well as Wang's  liberated  versions of the permutation groups, orthogonal groups and unitary groups, denoted $S_n^+, O_n^+ $ and $U_n^+$, respectively.
Note  that $SU_q(2)$ is coamenable for all $q$ but only of Kac type when $q=\pm1$, while the liberated quantum groups are all of Kac type, but only $O_2^+, S_2^{+}, S_3^+$ and $S_4^+$  are coamenable. 
In the context of the present paper, the most interesting quantum permutation group  is $S_4^+$, since $C(S_4^+)$ is an infinite-dimensional, non-commutative $C^*$-algebra while $C(S_n^+)=C(S_n)$ for $n\leq 3$, so that there are no non-classical quantum symmetries.   Note also that one has  $SU_{-1}(2)\simeq O_2^+$ and that {$S_4^+\simeq SO_{-1}(3)$} (a deformation of the classical Lie group $SO(3)$); see \cite{banica-bichon, banica-sym, Wang:quantum-symmetry} for more details.  All of these comments are meant to provide insight into why  these  are the examples that appear in Corollary \ref{introcor:examples} below. 

When $\GG$ is a matrix quantum group and the fundamental corepresentation $u$ is  equivalent to its conjugate corepresentation, one has that every $\alpha\in \Irred(\GG)$ appears  in some tensor power $u^{\boxtimes k}$, and $\GG$ thereby obtains a natural length function $\ell\colon \Irred(\GG)\to \NN_0$ by setting $\ell(\alpha)$ equal to the smallest such $k$. This is the situation when $\GG$ is $SU(2)$, $O_2^+$,  $SO(3)$ or $S_4^+$, and  for these examples we obtain quantum metric spaces:

\begin{introcorollary}\label{introcor:examples}
When equipped with the length function $\ell$ arising from the standard fundamental corepresentation, one has that $(C(SU(2)), L_\ell)$,   $(C(O_2^+),L_\ell)$,  $(C(SO(3)), L_\ell)$ and $(C(S_4^+), L_\ell)$ are compact quantum metric spaces. 
\end{introcorollary}

\subsubsection*{Notation} Throughout the article, we will assume that inner products are conjugate linear in the first argument. Moreover, $\odot$ will denote the algebraic tensor product of 
algebras, $\bar{\tens}$ is the von Neumann tensor product, and $\hat{\tens}$ is the (completed) tensor product of Hilbert spaces. We will denote by $\tens$ the minimal tensor product of $C^*$-algebras or the algebraic tensor product of vector spaces, and the usage will be clear from context. For indices $i$ and $j$ in some index set, we will denote by $\delta_{ij}$ the corresponding Kronecker delta.\\
Lastly, we will be using two versions of the Greek letter epsilon: the symbol $\epsilon$ will be reserved for the counit in Hopf algebras, while $\varps$ will be used to denote an arbitrarily small positive number.

\subsubsection*{Acknowledgements} The authors would like to thank Jens Kaad for his valuable insights and many interesting discussions, {
Christian Voigt for reminding them about the importance of understanding basic examples and Dimitris M.~Gerontogiannis and Bram Mesland for asking interesting questions and pointing out the  new approach in \cite{GerontogiannisMesland2025}.} Lastly, they would like to thank the anonymous referee for their thorough reading of the paper and many helpful comments.

The authors gratefully acknowledge the financial support from  the Independent Research Fund Denmark through grant no.~9040-00107B, 7014-00145B and 1026-00371B {and from the European Commision through the MSCA Staff Exchanges grant no.~101086394. The first named author was supported by The Research Council of Norway project 324944.

\section{Preliminaries}

\subsection{Compact quantum metric spaces}\label{sec:cqms}

The present section contains a brief introduction to Rieffel's theory of compact quantum metric spaces \cite{Rie:MSS, Rie:GHD,Rie:CQM},  which provides   an elegant non-commutative analogue of the classical theory of compact metric spaces.  The theory of compact quantum metric spaces can be based on either order unit spaces \cite{Rie:MSS}, $C^*$-algebras \cite{Li:ECQ,Lat:QGH} or operator systems \cite{Ker:MQG, Kerr-Li:GH}, among which we have chosen the latter approach for the sake of coherence with recent developments in \cite{KaadKyed2022, walter:GH-convergence, walter-connes:truncations}. To this end, we recall that a (concrete) \emph{operator system} is a self-adjoint, unital subspace $\X$ in a given unital $C^*$-algebra $A$. We will denote the closure of $\X$ by $X$ and refer to $\X$ as \emph{complete} if $\X=X$.
An element $x\in \X$ is \emph{positive} if this is the case in the ambient $C^*$-algebra $A$, and the \emph{state space} $\S(\X)$ is defined as the positive, unital functionals on $\X$. Note that $\S(\X)$ is compact for the weak$^*$ topology and homeomorphic to $\S(X)$ via the restriction map. 
\begin{dfn}[{\cite{Rie:matricial-bridges}}]
A \emph{slip-norm} on an operator system $\X$  is a seminorm $L\colon\X \to [0,\infty)$ satisfying that $L(1)=0$ and $L(x)=L(x^*)$ for all $x\in \X$.
\end{dfn}
Given a slip-norm $L$ on an operator system $\X$, one obtains an extended metric $d_L$ on $\S(\X)$ by setting
\[
d_L(\mu,\nu):=\sup\big\{ |\mu(x)-\nu(x)| \mid x\in \X, L(x)\leq 1 \big\}
\]
The adjective ``extended''  indicates that $d_L$ may attain the value infinity, but otherwise satisfies the usual properties of a metric.

\begin{dfn}
A \emph{compact quantum metric space} consists of a pair $(\X,L)$ where $\X$ is an operator system and $L\colon \X\to [0,\infty ) $ is a slip-norm such that $d_L$ 
metrizes
the weak$^*$ topology on $\S(\X)$. In this case, $L$ is referred to as a \emph{Lip-norm}.
\end{dfn} 
\begin{example}
The passage from classical compact metric spaces to their quantum analogues is obtained as follows: given a compact metric space $(M,d)$, one puts
\[
\X:=C_{\text{Lip}}(M):=\{f\in C(M) \mid f \text{ is Lipschitz continuous }\},
\] 
and defines $L_d\colon \X \to [0,\infty)$ by setting
\[
L_d(f):=\sup\left\{\frac{|f(a)-f(b)|}{d(a,b)} \mid a,b\in M, a\neq b \right\}.
\]
In this way, one obtains a compact quantum metric space. This was originally proven in \cite{KaRu:FSE,KaRu:OSC}; for a modern approach see \cite[Lemma 3.4]{Kaad2023}. Interesting non-commutative examples are obtained by considering a strongly continuous,  ergodic action $G\overset{\alpha}{\curvearrowright} A$ of a compact metric group $G$ on a unital $C^*$-algebra $A$ and defining
\[
L(a):=\sup\left\{\frac{\|\alpha_g(a)-a\|}{d_G(g,e)} \mid g\in G\setminus\{e\} \right\}
\]
on those elements for which the right hand side is finite \cite[Theorem 2.3]{Rie:MSA}. 
\end{example} 
\begin{remark}
As is customary in the field, we will often extend a slip-norm $L$ on an operator system $\X$ to the closure $X$ by declaring $L(x)=\infty$ for $x\in X\setminus \X$. 
We will further employ the standard abuse of notation and also refer to $(X,L)$ as a compact quantum metric space when $L \colon \X \to [0, \infty)$ is a Lip-norm.
\end{remark}

The theory of compact quantum metric spaces is heavily inspired by  Connes' non-commutative differential geometry \cite{Con:NCG}, in that many interesting examples arise from spectral triples.
For our purposes,  it is most convenient to choose a minimal set of axioms defining a spectral triple, and we will thus be ignoring gradings, regularity properties etc.  For background on non-commutative geometry, we refer to \cite{Con:NCG, elements-of-ncg}.

\begin{dfn}
A \emph{spectral triple} on a unital $C^*$-algebra $A$ consists of a representation\footnote{which we shall notationally suppress.} of $A$ on a Hilbert space $H$, a self-adjoint, unbounded operator $D$ on $H$ and a dense, unital $*$-subalgebra $\A\subseteq A$ satisfying that
\begin{itemize}
\item[(i)] Every $a\in \A$ preserves the domain of $D$ and $[D,a]=Da-aD$  extends to a bounded operator $\partial_D(a)$ on $H$.
\item[(ii)] The operator $D$ has compact resolvent. 
\end{itemize}
\end{dfn}
Given a spectral triple $(\A,H,D)$, one obtains a natural seminorm on $\A$ by setting $L_D(a):=\|\partial_D(a)\|$. In many interesting cases \cite{AgKa:PSM, Rie:GrpCstar, Christensen-Ivan:spectral-triples}, but not always \cite{KyedNest, PutJul:subshifts}, $L_D$ is a Lip-norm, in which case $(\A,H,D)$ is called a \emph{spectral metric space}. \\
To illustrate the complexity of the situation,  we now consider the fundamental example, dating back to Connes' seminal paper \cite{Con:CFH}, arising from a countable discrete group $\Gamma$ equipped with a proper length function $\ell\colon \Gamma \to [0,\infty)$. That is, the function $\ell$ satisfies
\begin{itemize}
\item[(i)] $\ell(x)=\ell(x^{-1})$ for all $x\in \Gamma$.
\item[(ii)] $\ell(xy)\leq \ell(x)+\ell(y)$ for all $x,y\in \Gamma$.
\item[(iii)] $\ell(x)=0$ if and only if $x=e$, and $\ell^{-1}([0,R])$ is finite for all $R\geq 0$ (properness).
\end{itemize}
When  $\Gamma$ is finitely generated, one obtains a length function by  fixing a finite, symmetric, generating set and setting $\ell(x)$ equal to the minimal length of a word in the generators expressing $x$. \\
From a length function $\ell$, one obtains a spectral triple $(\CC \Gamma, \ell^2(\Gamma),D_\ell)$ on the reduced group $C^*$-algebra $C^*_{\red}(\Gamma)$, by defining $D_\ell$ to be the closure of the unbounded operator $\mathcal{D}_\ell\colon \text{span}\{\delta_x \mid x\in \Gamma\} \to \ell^2(\Gamma)$ defined on the standard orthonormal basis  by $\mathcal{D}_\ell(\delta_x):=\ell(x)\cdot \delta_x$. Even in this case, it is very unclear when $L_{D_\ell}$ is a Lip-norm ---  at the time of writing, this is basically only known for groups of polynomial growth \cite{ChristRieffel2017} and hyperbolic groups  \cite{OzawaRieffel2005}. \\

We end this section with the primary tools available for verifying that a given pair $(\X,L)$, consisting of an operator system and a slip-norm thereon,  is a compact quantum metric space. The main result in this direction is due to Rieffel.
\begin{theorem}[{\cite[Theorem 1.8]{Rie:MSA}}]\label{thm:Rieffels-criterion}
The pair $(\X,L)$ is a compact quantum metric space if and only if the set $\{x\in \X\mid L(x)\leq 1\}$ projects to a totally bounded set in $\X/\CC$ via the quotient map $\X\ni x\mapsto [x] \in \X/\CC$.
\end{theorem}
Here, and throughout the text, we tacitly identify $\CC$ with the scalar multiples $\CC \cdot 1$ of the unit $1$ in the operator system $\X$, and the total boundedness is with respect to the quotient norm induced by the operator norm on $\X$.  Since $X$ is complete, this is equivalent to the set  $\{[x]\mid  x \in \X,  L(x)\leq 1\}$ having compact closure in  $X/\CC$.\\

The second criterion is due to Kaad, building, in turn,  on Rieffel's  theorem  above. To state this result,  it is convenient to first introduce the notion of \emph{finite diameter}. 
\begin{dfn}[\cite{Rie:MSS}]\label{def:finite-diameter}
The pair $(\X, L)$ is said to have \emph{finite diameter} if there exists a constant $C \geq 0$ such that 
\begin{align}\label{eq:diameter-ineq}
\Vert [x] \Vert_{\X/\CC} \leq C \cdot L(x) \quad \mbox{for all } x \in \C X.
\end{align}
\end{dfn}

Having finite diameter, in the sense just defined,  is equivalent to the (extended) metric $d_L$ assigning a finite diameter to the state space $\S(\X)$ \cite[Proposition 1.6]{Rie:MSA}. 
Note also, that by connectedness and compactness of $\S(\C X)$ in the weak$^*$  topology,   a compact quantum metric space automatically has finite diameter.
Kaad's criterion now reads as follows:

\begin{theorem}[{\cite[Theorem 3.1]{Kaad2023}}]\label{thm:kaad-criterion}
The pair $(\X,L)$ is a compact quantum metric space if and only if
\begin{itemize}
\item[(i)] $(\X,L)$ has finite diameter.
\item[(ii)] For every $\varps>0$ there exist an operator system $\X_\varps$ and unital, bounded maps $\iota, \Phi\colon \X \to \X_\varps$ such that $\iota$ is an isometry,  $\Phi$ is positive with finite-dimensional image and  $\|\iota(x)-\Phi(x)\|\leq \varps\cdot  L(x)$ for all $x\in \X$.

\end{itemize}
\end{theorem}

\subsection{Quantum groups}\label{sec:qgrps}
In this section, we survey the relevant background results from the theory of compact quantum groups, following the $C^*$-algebraic approach initiated by Woronowicz \cite{wor:cpqgrps}; for detailed expositions, see  \cite{KlSc:QGR} or \cite{Timmermann-book}. A compact quantum group is a  pair $(A, \Delta)$ consisting of a unital $C^*$-algebra $A$ and a unital $*$-homomorphism $\Delta\colon A\to A \tens A$ satisfying the coassociativity condition $(\Delta \tens \id)\Delta=(\id \tens \Delta )\Delta$ and such that the sets $\Delta(A)(1\tens A)$ and $\Delta(A)(A\tens 1)$ span dense subspaces in $A\tens A$.  Recall that here, and below,``$\tens$'' is used to denote the minimal tensor product of $C^*$-algebras. 

\begin{example}
 Given a compact Hausdorff group $G$, setting $A=C(G)$ and letting $\Delta$ be the map dual to the group multiplication, one obtains  a compact quantum group. Even better, all commutative examples take this form,  making the theory of compact quantum groups perfectly compatible with Gelfand duality. 
 Non-commutative examples may be constructed from a discrete group $\Gamma$, by setting $A=C^*_\red(\Gamma)$ and $\Delta(\lambda_\gamma)=\lambda_\gamma\tens \lambda_\gamma$, where $\lambda$ denotes the left regular representation of $\Gamma$.  The quantum group literature contains numerous interesting examples not arising from compact or discrete groups, for instance Woronowicz' quantum $SU(2)$ \cite{Wor:UAC} and Wang's liberations of the unitary groups, the orthogonal groups and the permutation groups \cite{Wang:free-products, Wang:quantum-symmetry}.
 
 \end{example}
  One of the main features of the theory is the existence of a unique bi-invariant state $h\colon A\to \CC $, meaning that 
\begin{align}\label{eq:bi-inv}
(h\tens \id)\Delta(a)=h(a)1=(\id\tens h)\Delta(a)
\end{align}
 for all $a\in A$. In the commutative situation, where $A=C(G)$, $h$ is given by integration against the unique Haar probability measure $\mu$ and the bi-invariance property \eqref{eq:bi-inv} is equivalent to the translation invariance of $\mu$. For this reason, $h$ is dubbed the \emph{Haar state} also in the non-commutative setting. 
  When $A=C^*_\red(\Gamma)$, the Haar state is simply the canonical trace. However,   in general $h$ need not be a trace and when it is, the compact quantum group $\GG$ is said to be of \emph{Kac type.}\\
  Even in the non-commutative situation, the notation is often chosen to reflect the fact that one thinks of a compact quantum group as a  ``non-commutative algebra of continuous functions'' on a (non-existing) quantum object $\GG$.   With this in mind, we therefore introduce the following standard notation for the objects arising from the GNS construction associated with $h$: 
  
  \begin{enumerate}
\item[-] The Hilbert space is denoted by $L^2(\GG)$, {its norm by $\|\cdot \|_2$,} and the canonical map from $A$ to $L^2(\GG)$ is denoted by $\Lambda$.
\item[-] The GNS representation is denoted by $\lambda\colon A\to \BB(L^2(\GG))$ and its image $\lambda(A)$ by $C(\GG)$.
\item[-] The von Neumann algebra generated by $C(\GG)\subseteq \BB(L^2(\GG))$ is denoted by $L^\infty(\GG)$.
\end{enumerate} 
The comultiplication  descends to the $C^*$-algebra  $C(\GG)$, turning it into a compact quantum group in its own right, whose Haar state is the vector state 
$\lambda(a)\mapsto \left\langle \Lambda(1), \lambda(a) \Lambda(1) \right\rangle$. 
In what follows, we will only be concerned with this represented version of the quantum group, and therefore also denote the aforementioned vector state by $h$, the induced comultiplication by $\Delta$ and the canonical injection $C(\GG) \to L^2(\GG)$ by $\Lambda$.  Whenever notationally convenient, we shall suppress the map $\Lambda$ and simply view 
$C(\GG)$ as a subspace of $L^2(\GG)$.

\subsubsection{The fundamental unitaries}\label{sec:fund-unit}
On $L^2(\GG)\hat{\tens}L^2(\GG)$, we have the left- and right fundamental unitaries, $W$ and $V$, defined on the dense subset $\Lambda\tens \Lambda(C(\GG) \odot C(\GG))$ by the relations
\[
W^*(\Lambda(x) \tens \Lambda(y))=\Lambda\tens \Lambda\big(\Delta(y)(x\tens 1)\big) \quad \text{ and } \quad V(\Lambda(x) \tens \Lambda(y))=\Lambda\tens \Lambda\big( \Delta(x)(1\tens y)\big).
\]
They both implement the comultiplication on $C(\GG)$ by means of the formula
\begin{align}\label{eq:implementation-eq}
\Delta(x)=W^*(1\tens x)W =V(x\tens 1) V^*.
\end{align}
The comultiplication therefore extends to a normal  $*$-homomorphism $\Delta\colon L^\infty(\GG) \to L^\infty(\GG)\bar{\tens}L^\infty(\GG)$ (still implemented by $W$ and $V$) turning $L^\infty(\GG)$ into a compact von Neumann algebraic quantum group \cite{Kustermans-Vaes:vNA, Timmermann-book}.  
Following the established conventions in the literature, we think of $(C(\GG), \Delta)$ and $(L^\infty(\GG), \Delta)$ as different operator algebraic realizations of the (non-existing) underlying quantum group $\GG$, which is indicated linguistically by phrases such as ``Let $\GG$ be a compact quantum group....''.

\subsubsection{Flips and leg-numbering notation} On the Hilbert space $L^2(\GG)\hat{\tens} L^2(\GG)$, we will be using the flip unitary $\Sigma$, defined by $\Sigma( \Lambda(x)\tens \Lambda (y))=\Lambda(y)\tens \Lambda(x)$. Note that conjugation with $\Sigma$ implements the corresponding flip map $\sigma$ at the $C^*$-algebraic level in the sense that 
\[
\sigma (a\tens b):=b\tens a=\Sigma(a\tens b)\Sigma,
\]
for $a,b\in \BB(L^2(\GG))$. 
We will also be using the standard \emph{leg-numbering notation},  meaning that for an operator $T\in \BB(L^2(\GG) \hat{\tens}L^2(\GG))$, $T_{12}$ denotes $T\tens 1$, $T_{13}$ denotes  $(1\tens \sigma)(T\tens 1)$, etc. 
Lastly, whenever $a\in C(\GG)$ satisfies that $\Delta(a)\in C(\GG) \odot C(\GG)$ we will 
frequently
use the \emph{Sweedler notation} and write $\Delta(a)=a_{(1)}\tens a_{(2)}$ for ease of notation. A similar remark applies to more general coactions of $\GG$, to be introduced in Section \ref{sec:coact} below.

\subsubsection{Corepresentation theory}\label{sec:corep}
A \emph{unitary corepresentation} of a compact quantum group $\GG$ on a Hilbert space $H$ is a unitary $U\in L^\infty(\GG)\bar{\tens}\BB(H)$ satisfying 
\begin{equation}\label{eq:corep}
(\Delta \tens \id )(U)=U_{13}U_{23}.
\end{equation}
Both $W$ and $V$ satisfy the so-called \emph{pentagon identity}  (stating that $W_{12}W_{13}W_{23}=W_{23}W_{12}$ in the case of $W$) which, combined with the formulas \eqref{eq:implementation-eq}, implies that $W$ and $\Sigma V\Sigma$ are both unitary corepresentations. \\
The  corepresentation theory of $\GG$ mirrors, in many ways, the representation theory of a classical compact group, in that one has  natural notions of direct sums, tensor products, equivalence and irreducibility.  Moreover,  every finite-dimensional unitary corepresentation decomposes as a direct sum of irreducible ones. We let $I=\Irred(\GG)$ denote the set of equivalence classes of irreducible corepresentations and let $(u^{\alpha})_{\alpha \in I}$ denote a fixed set of representatives for these equivalence classes.  The unit in $C(\GG)$ is a 1-dimensional (irreducible) corepresentation, and we denote its class in $\Irred(\GG)$ by $e$. 
Each $u^{\alpha}$ is an element in $L^\infty(\GG) \bar{\tens} \BB(H_\alpha)$ for a finite-dimensional Hilbert space $H_\alpha$ (say, of dimension $d_\alpha)$, and upon choosing an orthonormal basis for $H_{\alpha}$ we obtain an identification $L^\infty(\GG) \bar{\tens} \BB(H_\alpha)=\mathbb{M}_{d_\alpha}(L^\infty(\GG))$, and may thus also consider the associated matrix coefficients $u_{ij}^\alpha\in L^\infty(\GG)$. It turns out that $u_{ij}^\alpha \in C(\GG)$ and that the set
\[
\Pol(\GG):=\text{span}_{\CC}\{u_{ij}^\alpha \mid \alpha \in I, 1\leq i,j \leq d_\alpha\}
\]
forms a dense Hopf $*$-subalgebra in $C(\GG)$.
 More precisely, $\Delta$ restricts to a comultiplication $\Delta\colon \Pol(\GG) \to \Pol(\GG) \odot \Pol(\GG)$ given by
$
\Delta(u_{ij}^\alpha)= \sum_{k=1}^{d_\alpha} u_{ik}^\alpha \tens u_{kj}^\alpha
$, and  $\Pol(\GG)$  can be further equipped with an antipode $S\colon \Pol(\GG) \to \Pol(\GG)$ and a counit $\epsilon\colon \Pol(\GG) \to \CC$; \cite[Theorem 5.4.1]{Timmermann-book}.
For future reference, we note that since $\Pol(\GG)$ is a Hopf $*$-algebra, $W^*$ and $V$ actually map $\Pol(\GG)\odot \Pol(\GG)$ bijectively onto itself (see e.g.~\cite[Theorem 1.3.18]{Timmermann-book}) and the same is therefore true for $W$ and $V^*$. \\
Direct sums and tensor products 
of corepresentations are denoted $\oplus$ and $\boxtimes$, respectively. For $\alpha \in \Irred(\GG)$, the conjugate corepresentation is defined as $\overline{u^{\alpha} }:=((u_{ij}^\alpha)^*)$. This still satisfies the corepresentation relation \eqref{eq:corep} but may fail to be a unitary matrix ---  
it is, however,  equivalent to a unique element in $\Irred(\GG)$ which we will denote by $\bar{\alpha}$; see e.g.~\cite[Theorem 5.3.3 and Corollary 5.3.10]{Timmermann-book}. Lastly, we denote by $N_{\alpha, \beta}^\gamma$ the multiplicity of $\gamma \in \Irred(\GG)$ in the decomposition  of $\alpha \boxtimes \beta$ into irreducibles; i.e.
\[
\alpha \boxtimes \beta \simeq \bigoplus_{\gamma \in \Irred(\GG)} \gamma^{\oplus N_{\alpha, \beta}^\gamma}
\]
In the case of a classical compact group, the above constructions of course agree with their classical counterparts. For examples of the form $C^*_\red(\Gamma)$, all irreducible corepresentations are 1-dimensional and $(\lambda_\gamma)_{\gamma \in \Gamma}$ is a complete set of representatives for $\Irred(\GG)$, and $\Pol(\GG)$ agrees with the group algebra $\CC\Gamma \subset C^*_\red(\Gamma)$.

\subsubsection{Coactions}\label{sec:coact}
A (left) \emph{coaction} of $\GG$ on a Hilbert space $H$ is a normal, unital, injective $*$-homomorphism $\delta\colon \BB(H)\to L^\infty(\GG)\bar{\tens} \BB(H)$ satisfying $(\id\tens\delta)\delta=(\Delta\tens \id)\delta$, see  \cite{Vaes:implementation}. For every unitary corepresentation $U\in L^\infty(\GG)\bar{\tens} \BB(H)$, one obtains an associated coaction $\delta\colon \BB(H)\to L^\infty(\GG)\bar{\tens} \BB(H)$ by setting $\delta(T)=U^*(1\tens T)U$.

\subsubsection{Coamenability}
A compact quantum group $\GG$ is said to be \emph{coamenable}, if the counit $\epsilon \colon \Pol(\GG)\to \cc$ extends to a character on $C(\GG)$. In the case of a classical discrete group $\Gamma$, the counit $\epsilon \colon \cc \Gamma \to \cc$ is simply the extension of the trivial representation, so in this situation the notion of coamenability agrees with amenability of the group $\Gamma$, see e.g.~\cite{Brown-Ozawa}.
The notion of coamenability has been studied extensively, and may be characterised in a number of different ways; see e.g.~\cite{tomatsu-amenable, bedos-murphy-tuset} for details. \\
Other examples of coamenable quantum groups include all commutative examples (where the counit is given by evaluation at the neutral element in the underlying compact group), Woronowicz's quantum $SU(2)$ \cite{banica-subfactor}, and the quantum permutation group $S_4^+$ on four elements \cite{banica-sym}.

\subsubsection{The fusion algebra}\label{sec:fusion-alg}
Letting $F(\GG)$ denote the formal $\CC$-vector space with basis $\Irred(\GG)$, we obtain a $*$-algebra called the \emph{fusion algebra},  in which the sum is induced by the {direct sum}, the involution is induced by conjugation $\alpha \mapsto \bar{\alpha}$ and the product is induced by (decomposition of) the tensor product of corepresentations: 
\[
\alpha \cdot \beta :=\sum_{\gamma \in \Irred(\GG)} N_{\alpha, \beta}^\gamma \gamma.
\]
Note that the equivalence class, $e$, of the unit in $C(\GG)$ serves as unit in $F(\GG)$.
In the case where $\GG$ is a classical compact group, we recover the usual fusion algebra, and when $C(\GG)=C^*_\red(\Gamma)$ the fusion algebra agrees with the group algebra $\CC\Gamma$. 
For more examples, see Section \ref{sec:examples}.

\begin{remark}
	Note that if two quantum groups $\GG_1$ and $\GG_2$ are monoidally equivalent (see e.g. \cite[Definition 2.3.8]{NeshveyevTusetBook}), then the equivalence induces a bijection between $\Irred (\GG_1)$ and $\Irred(\GG_2)$ in such a way that the vector space isomorphism $F(\GG_1) \cong F(\GG_2)$ is a $*$-homomorphism. 
\end{remark}

\subsubsection{The algebra of central functions}
In general, $\sigma\circ \Delta \neq \Delta$ which leads one to study the \emph{central functions} on $\GG$ defined, at the von Neumann algebraic level, as
\[
L_z^\infty(\GG):=\{a\in L^\infty(\GG) \mid \sigma(\Delta(a))=\Delta(a)\}.
\]
Similarly, one puts $C_z(\GG):=C(\GG)\cap L_z^\infty(\GG) $ and $\Pol_z(\GG):=\Pol(\GG)\cap L_z^\infty(\GG)$. In the case of a classical compact group $G$, the algebra $C_z(G)$ consists of functions invariant under the conjugation action; i.e.~the so-called class functions.  \\
The character map $\chi\colon F(\GG) \to \Pol_z(\GG)$, given by $\chi(\alpha)=\sum_{i=1}^{d_\alpha} u^\alpha_{ii}$, extends by linearity to a unital $*$-isomorphism 
\cite[Proposition 3.23 and Theorem 3.24]{BanicaIntro2022},
and in this way we may view $F(\GG)$ as a subalgebra of $\Pol(\GG)$ whenever convenient.  Moreover, it holds that $\Pol_z(\GG)$ is dense in $C_z(\GG)$  and $L_z^\infty(\GG)$ in the norm- and strong operator topology, respectively, as was  shown in \cite{Crann:CaracterDensity}.   One has an $h$-preserving conditional expectation $E\colon C(\GG)\to C_z(\GG)$ exactly when $\GG$ is of Kac type (i.e.~when $h$ is a trace) \cite[Lemma 6.3]{Wang:lacunary}, and this conditional expectation will play a key role in the proof of Theorem \ref{introthm:lifting}.

\subsubsection{The dual}
Associated with $\GG$ is  its dual \emph{discrete quantum group}, $\hat{\GG}$, \cite[Section 3.3]{Timmermann-book}, which also allows for a Hopf-algebraic, a $C^*$-algebraic and a von Neumann algebraic description. We denote the three algebras as
\[
c_c({\hat{\GG}}) \subset c_0(\hat{\GG}) \subset \ell^\infty(\hat{\GG}) \subset \BB(L^2({\GG})),
\]
and note that when $\GG$ is a classical, compact, \emph{abelian} group then the algebras above agree with those functions on the (discrete) Pontryagin dual $\hat{G}$, that are compactly supported, vanishing at infinity, and bounded, respectively. This provides a full generalisation of Pontryagin duality, in that $\hat{\hat{\GG}}\simeq \GG$. Moreover, when starting with a (potentially non-abelian) discrete group $\Gamma$, the $C^*$-algebraic discrete quantum group  $ c_0(\Gamma)$ is exactly the dual of $C^*_\red(\Gamma)$. \\
The multiplicative unitaries, $W$ and $V$, introduced above are closely related to the dual pair $(\GG, \hat{\GG})$, {as one has that}
\begin{align}\label{eq:mult-unit}
W\in L^\infty(\GG) \bar{\otimes} \ell^\infty(\hat{\GG}) \quad \text{and} \quad V\in \ell^\infty(\hat{\GG})'\bar{\tens} L^\infty(\GG)   
\end{align}

\begin{remark}
The main  focus in the present paper is on quantum metric structures on the $C^*$-algebra $C(\GG)$. If $C(\GG)$ admits the structure of a compact quantum metric space, then $C(\GG)$ is necessarily separable, as is seen from the following standard argument:
 if the state space $\S(C(\GG))$ is 
 metrized, 
 it is itself separable and hence so is $C(\S(C(\GG)))$. The positive elements $C(\GG)_+$ embed isometrically into $C(\S(C(\GG)))$ as evaluation functionals, so also $C(\GG)_+$  is separable. Since every element in $C(\GG)$ is a linear combination of four positive elements, $C(\GG)$  is separable.
We will therefore enforce separability of $C(\GG)$ as a standing  assumption throughout the rest of the paper, without further mention. Note that this implies that $L^2(\GG)$ is a separable Hilbert space and that $\Irred(\GG)$ is at most countable.

\end{remark}

\section{Dirac operators from length functions}
This section is devoted to a study of the spectral triples arising from a length function on a compact quantum group. Length functions on classical discrete  groups  have played a prominent role in non-commutative geometry since Connes' seminal paper \cite{Con:CFH}, and have also been studied in detail from the quantum metric point of view in \cite{OzawaRieffel2005, ChristRieffel2017, Rie:GrpCstar}.
Length functions have also   appeared in the quantum group literature  in connection with  the rapid decay  property \cite{Vergnioux:rapid, BVZ:rapid-decay}. 
We recall the definition here:

\begin{dfn}\label{def:length-function}
Let $\GG$ be a compact quantum group. A {\emph{length function}} on $\GG$ is a function $\ell\colon \Irred(\GG) \to [0,\infty)$ such that

\begin{itemize}
\item[(i)] $\ell(e)=0$.
\item[(ii)] $\ell(\bar{\alpha})=\ell(\alpha)$ for all $\alpha\in \Irred(\GG)$.
\item[(iii)] $\ell(\gamma)\leq \ell(\alpha)+\ell(\beta)$ for all $\alpha, \beta,\gamma\in \Irred(\GG)$ such that $\gamma$ is equivalent to a sub-corepresentation of $\alpha\boxtimes \beta$.
\end{itemize}
A length function $\ell$ is called \emph{proper} if $\ell^{-1}([0,R])$ is finite for all $R\geq 0$ and $\ell(\alpha)=0$ only when $\alpha=e$.
\end{dfn}

The standard example of a length function, and indeed the only one for which we will provide examples in Section \ref{sec:examples}, arises when $\Irred(\GG)$ is finitely generated.  This, by definition, means that there exists a finite set $S\subset \Irred(\GG)$ such that every $\alpha\in \Irred(\GG)$ is equivalent to a sub-corepresentation in a tensor product of elements from $S$. Setting $\ell(\alpha)$ equal to the smallest integer $k$ such that $\alpha$ is equivalent to a sub-corepresentation of a tensor product of $k$ elements from $S$ provides a length function on $\GG$, commonly referred to as the \emph{word length function} (with respect to $S$). In Section \ref{sec:examples}, our primary interest will be in the situation where $S$ can be chosen to consist of single element $u$ (this is for instance the case for quantum $SU(2)$), and in this situation $\ell(\alpha)$ is simply the smallest $k$ for which $\alpha$ appears in the decomposition of the tensor power $u^{\boxtimes k}$.  For this reason, and to formally conform with the setup in \cite{OzawaRieffel2005}, we will now restrict to proper length functions with values in $\NN_0$.

From a proper length function, one obtains an unbounded operator $\mathcal{D}_\ell\colon \Lambda(\Pol(\GG))\to L^2(\GG)$  by setting
\begin{align}\label{eq:Dirac-defi}
\mathcal{D}_\ell(\Lambda(u_{ij}^\alpha)):=\ell(\alpha)\cdot\Lambda(u_{ij}^\alpha).
\end{align}
As the set {$\{\tfrac{1}{\|\Lambda(u_{ij}^\alpha)\|_2}\Lambda(u^\alpha_{ij})\}$} forms an orthonormal basis for $L^2(\GG)$, this defines an essentially self-adjoint operator and we denote its self-adjoint closure by $D_\ell$. In analogy with the
case of classical discrete groups, one obtains a non-commutative geometry in this way. 
Note that the following result is stated only for word length functions in \cite[Lemma 5.4]{BVZ:rapid-decay}, but their proof only depends on  \cite[Lemma 1.1]{OzawaRieffel2005}, so it also holds true in our setting.
\begin{lemma}[{\cite[Lemma 5.4]{BVZ:rapid-decay}}]
The data $(\Pol(\GG),L^2(\GG), D_\ell)$ forms a spectral triple.
\end{lemma}

As already mentioned in the introduction, the main purpose of the present paper is to investigate the quantum metric structure associated with the commutator seminorm arising from this construction. This idea was already investigated in \cite{BVZ:rapid-decay}, where the authors prove that in the presence of the {rapid decay property}, one does indeed obtain a compact quantum metric structure by taking \emph{iterated} commutators with $D_\ell$, where the number of iterations depends on the degree of rapid decay. 
To the best of our knowledge,  \cite[Theorem 7.4]{BVZ:rapid-decay}  is the only positive result in this direction, but even for reasonably tame examples, such as  $SU_q(2)$, their result seems to require taking at least two commutators before yielding a Lip-norm. In particular,  there are no known examples of length functions on (genuine) quantum groups for which the basic commutator seminorm gives a compact quantum metric structure. 
Iterated commutator seminorms have the drawback that they will rarely satisfy the Leibniz rule, thus straying from the original noncommutative geometric motivation behind the problem. The aim in the present paper is to provide the first examples for which a single commutator suffices; see Corollary \ref{introcor:examples}. 
To this end, the fusion algebra will play a central role and in the following section we will show how one may also base a spectral triple on this algebra.

\subsection{A spectral triple for the fusion algebra}\label{sec:spec-trip-fusion}
Consider again a compact quantum group $\GG$ and its fusion algebra $F(\GG)=\T{span}_{\CC}\{\alpha \mid \alpha\in \Irred(\GG)\}$. We endow $F(\GG)$ with the inner product for which the elements in $\Irred(\GG)$ form an orthonormal basis, and denote the completion by $\ell^2(\Irred(\GG))$, and by $c_c(\Irred(\GG))$ the dense subspace spanned by the orthonormal basis; i.e.~the image of $F(\GG)$ in $\ell^2(\Irred(\GG))$. For each $x\in F(\GG)$, we therefore obtain a linear map ${\pi}_0(x)\colon c_c(\Irred(\GG)) \to c_c(\Irred(\GG))$ which multiplies from the left with $x$, and we now show that this map extends boundedly to the Hilbert space $\ell^2(\Irred(\GG))$.  Recall that $\lambda$ denotes the GNS representation of $C(\GG)$ on $L^2(\GG)$.
 
 \begin{lemma}\label{lem:intertwining-lemma}
 The character map $\chi \colon c_c(\Irred(\GG)) \to \Lambda(\Pol_z(\GG))$, $\alpha\mapsto \sum_{i=1}^{d_\alpha} \Lambda(u_{ii}^\alpha)$, extends to a unitary $\tilde{\chi}\colon \ell^2(\Irred(\GG))\to L^2_z(\GG):= \overline{\Lambda(\Pol_z((\GG))}\subset L^2(\GG)$. For $x\in F(\GG)$, it holds that
  \[
{\pi}_0(x) =    \tilde{\chi} ^*\lambda(\chi(x))\tilde{\chi},
 \]
  as operators on $c_c(\Irred(\GG))$,  and ${\pi}_0(x)$ therefore extends to a bounded operator  ${\pi}(x)$ on $\ell^2(\Irred(\GG))$ which equals
  $\tilde{\chi}^*\lambda(\chi(x))\tilde{\chi}$.
 \end{lemma}
 \begin{proof}
By definition, the image of $\chi$ is a dense subspace in $L^2_z(\GG)$. Moreover,  by the Schur orthogonality relations \cite[Proposition 5.3.8 (iii)]{Timmermann-book}, we obtain that 
\[
\inn{ \chi({\alpha}), \chi(\beta) }=\sum_{i=1}^{d_{\alpha}}\sum_{j=1}^{d_\beta} h((u_{ii}^\alpha)^* u_{jj}^\beta))=\delta_{\alpha,\beta} \sum_{i=1}^{d_\alpha} h( (u_{ii}^\alpha)^* u_{ii}^\alpha)=\delta_{\alpha, \beta}.
\]
 In other words, $\chi$ maps the orthonormal basis onto an orthonormal  and spanning subset of $L^2_z(\GG)$ and hence extends to a unitary operator as claimed.  For the intertwining statement, just note that
for $\alpha,\beta\in \Irred(\GG)$ one has
\[
\tilde{\chi} {\pi}_0(\beta) (\alpha)= \Lambda(\chi(\beta \cdot \alpha))=\Lambda(\chi(\beta)\chi(\alpha))=\lambda(\chi(\beta))\tilde{\chi}(\alpha),
\]
from which the general case follows by linearity. 
 \end{proof}
 
 \begin{remark}\label{rem:GNS-bd}
 The Haar state, $h$, pulls back to a positive functional $\tau\colon F(\GG) \to \CC$ via the character map $\chi$, which on the basis $\Irred(\GG)$ is given by $\tau(\alpha)=\delta_{\alpha,e}$. By Frobenius reciprocity, the map $\tau$ is a trace and hence so is the restriction of $h$ to $\Pol_z(\GG)$.  Also note that the representation ${\pi}$ constructed above agrees with the GNS representation arising from $\tau$,
  and the latter therefore, in particular, gives rise to bounded operators; see also \cite{Hiai-Izumi} for a more general approach to this.
  \end{remark}
   We now furthermore assume that $\GG$ is endowed with a proper length function $\ell\colon \Irred(\GG) \to \NN_0$. 
  Setting
$A_n:=\T{span}\{\alpha \mid \ell(\alpha)\leq n\}$ provides a filtration of $F(\GG)$, in the sense that $A_0 = \CC\cdot e$, $A_n^*=A_n$ and $A_n\cdot A_m\subseteq A_{n+m}$, which follows directly from the defining properties of $\ell$.
 These comments, together with Remark \ref{rem:GNS-bd}, show that the fusion algebra $F(\GG)$ falls within the class of filtered $*$-algebras studied by Ozawa and Rieffel in  \cite{OzawaRieffel2005}.\\
The length function  also defines an unbounded operator $\tilde{\mathcal{D}}_\ell\colon c_c(\Irred(\GG)) \to \ell^2(\Irred(\GG))$ given by $\tilde{\mathcal{D}}_\ell(\alpha)=\ell(\alpha)\cdot \alpha$.
 This operator is essentially self-adjoint, and we denote its self-adjoint closure by $\tilde{{D}}_\ell$.
 \begin{lemma}
 Let $\GG$ be a compact quantum group equipped with a proper length function $\ell$. Then 
 the
 data $\big(F(\GG), \ell^2(\Irred(\GG)), \tilde{D}_{\ell}\big)$ is a spectral triple.
 \end{lemma}
 \begin{proof}
As noted above, our setup is compatible with that studied   in \cite{OzawaRieffel2005} and the fact that $\tilde{D}_\ell$ has bounded commutators with elements from $F(\GG)$ therefore follows from  \cite[Lemma 1.1]{OzawaRieffel2005}. That $\tilde{D}_\ell$ has compact resolvent is trivial if $\Irred(\GG)$ is finite and follows from the properness of $\ell$ when $\Irred(\GG)$ is infinite.
 \end{proof}

 By construction,  the unbounded operator  $\mathcal{D}_\ell$  on $L^2(\GG)$, defined in \eqref{eq:Dirac-defi}, preserves $\Lambda(\Pol_z(\GG))$. Moreover, the restriction $\mathcal{D}_\ell^z$, of $\mathcal{D}_\ell$ to $\Lambda(\Pol_z(\GG))$  is essentially self-adjoint, and we denote its self-adjoint closure by ${D}_\ell^z$. 
 In this way, $(\Pol_z(\GG), L^2_z(\GG), D_{\ell}^z)$ becomes a spectral triple 
  and we have:
 \begin{prop}\label{prop:unitarily-equivalent}
 The spectral triples  $\big(F(\GG), \ell^2(\Irred(\GG)),\tilde{D}_{\ell} \big)$ and $\big(\Pol_z(\GG), L^2_z(\GG),D_{\ell}^z\big)$ are unitarily equivalent via the unitary $\tilde{\chi}$. In particular,  if one is a spectral metric space, then so is the other. 
 \end{prop}
 \begin{proof}
 Note that $\tilde{\chi}$ maps $c_c(\Irred(\GG))$ onto $\Lambda(\Pol_z(\GG))$, and
for $\alpha\in \Irred(\GG)$ we have
 \[
 \tilde{\chi} \tilde{D}_{\ell} ({\alpha})=  \ell({\alpha}) \Lambda({\chi(\alpha)})= D_{\ell}^z \tilde{\chi}(\alpha),
 \]
so that $\tilde{\chi}$ intertwines  the two Dirac operators on their cores.  From this it easily follows that $\tilde{\chi}(\T{Dom}(\tilde{D}_\ell))\subset \T{Dom}(D^z_\ell)$ and that $\tilde{\chi}\tilde{D}_\ell=D_\ell^z \tilde{\chi}$. 
 \end{proof}
 
 
 \begin{cor}\label{cor:fusion-vs-central}
 If   $\big(F(\GG), \ell^2(\Irred(\GG)),\tilde{D}_{\ell} \big)$ is a spectral metric space then the restriction of $L_\ell$ provides $C_z(\GG)$ with the structure of a compact quantum metric space.
 \end{cor}
 \begin{proof}
 We first note that,  for $a\in \Pol_z(\GG)$, we have
 \begin{align*}
 L_\ell(a) = \|[D_\ell, a]\|
& =\sup\{ \| (D_\ell a -a D_\ell)\xi \| \ \mid \xi\in (\Lambda \Pol(\GG))_1 \}\\
 &\geq  \sup\{ \| (D_\ell a -a D_\ell)\xi \| \ \mid \xi\in (\Lambda \Pol_z(\GG))_1 \}\\
 &=  \sup\{ \| (D_\ell^z a -a D_\ell^z)\xi \| \ \mid \xi\in (\Lambda \Pol_z(\GG))_1 \}= \|[D_\ell^z, a]\|
 \end{align*}
 By assumption, 
 $\big(F(\GG), \ell^2(\Irred(\GG)),\tilde{D}_{\ell} \big)$
 is a spectral metric space and, by Proposition \ref{prop:unitarily-equivalent}, the same is therefore true for $(\Pol_z(\GG), L^2_z(\GG),D_{\ell}^z)$. By Rieffel's criterion  (Theorem \ref{thm:Rieffels-criterion}), 
  this is equivalent to the Lipschitz unit ball $\{a\in \Pol_z(\GG) \mid \|[D_\ell^z, a]\|\leq 1\}$ having pre-compact image in $C_z(\GG)/\CC$. By the inequality established above, 
 \[
 \{a\in \Pol_z(\GG) \mid L_\ell(a)\leq 1\}  \subseteq  \{a\in \Pol_z(\GG) \mid \|[D_\ell^z, a]\|\leq 1\}, 
 \]
 and the former set therefore has pre-compact image in  $C_z(\GG)/\CC$ as well. Applying Rieffel's criterion  again, we conclude that $(C_z(\GG), L_\ell )$ is a compact quantum metric space. 
 \end{proof}

\section{The conjugation coaction}
For a  classical compact group $G$,  the algebra $C_z(G)$ consists of the conjugation invariant functions, i.e.~those  $f\in C(G)$ for which $ f(g^{-1}xg)=f(x)$  for all  $g,x\in G$. The algebra $C_z(G)$ may therefore be viewed as the fixed point algebra for the conjugation action $G\overset{\alpha} \curvearrowright C(G)$, and one obtains a natural conditional expectation $E\colon C(G)\to C_z(G)$  by 
integrating  the action against the Haar probability measure $\mu$ on $G$:
\begin{align}\label{eq:cond-exp-integral}
E(f)=\int_{G} \alpha_g(f) \ d\mu(g)
\end{align}
In this section, we  show how this situation carries over to the setting of compact quantum groups of Kac type, and prove that the  conditional expectation obtained is a contraction for the seminorm $L_\ell$ arising from a proper length function. It is quite likely that most of the results in the present section are known to experts in the field, but since we were unable to find explicit references in the literature, we provide the details for the convenience of the reader. \\

The key tool will be the conjugation coaction studied in \cite{Crann:CaracterDensity, Crann:InnerAmenability}, and we will therefore be following their conventions below.
It is not hard to see that both $W$ and $\Sigma V\Sigma$ are unitary corepresentations, where $W$ and $V$ are the fundamental unitaries from Section \ref{sec:fund-unit}. From \eqref{eq:mult-unit}, it follows that $W_{23}$ and $(\Sigma V\Sigma)_{13}$ commute and hence also 
 $Z:=W\Sigma V\Sigma \in L^\infty(\GG) \bar{\tens} \BB(L^2(\GG)) $ becomes a unitary corepresentation.  The associated coaction $\delta\colon \BB(L^2(\GG)) \to L^\infty(\GG) \bar{\tens} \BB(L^2(\GG)) $,
\[
\delta(T):=Z^*(1\tens T)Z,
\]
is going to be the quantum analogue of the conjugation action on a classical group discussed above, and we will therefore refer to it as the \emph{conjugation coaction}.
We denote by $\BB(L^2(\GG))^{\GG}$ the corresponding \emph{fixed point algebra}
\[
\{T\in \BB(L^2(\GG)) \mid \delta(T)=1\tens T\}.
\]

\begin{lemma}\label{lem:fixed-pt-lem}
For $a\in L^\infty(\GG)$, one has $\delta(a)=1\tens a$ if and only if $a\in L^\infty_z(\GG)$.
\end{lemma}
\begin{proof}
This follows from \eqref{eq:implementation-eq} via the following computation
\begin{align*}
\delta(a) =1\tens a &\Longleftrightarrow \Sigma V^* \Sigma W^*(1\tens a)W\Sigma V\Sigma=1\tens a \\
&  \Longleftrightarrow  \Delta(a) =\Sigma V(a\tens 1)V^*\Sigma \\
&\Longleftrightarrow \Delta(a)=\sigma \Delta(a)  \qedhere
\end{align*}
\end{proof}

By  \cite[Lemma 6.3]{Wang:lacunary}, there exists an $h$-preserving conditional expectation from $C(\GG)$ onto $C_z(\GG)$ exactly when $\GG$ is of Kac type. 
In order to show that this conditional expectation is a contraction for seminorms of the form $L_\ell$, we will need a quantum analogue of the formula \eqref{eq:cond-exp-integral},  which is obtained in the following proposition.

\begin{prop}\label{prop:cond-ext-prop}
The map $E\colon \BB(L^2(\GG)) \to \BB(L^2(\GG))$ given by  $E(a):=(h\tens \id)\delta(a)$ defines a normal conditional expectation onto the fixed point algebra $\BB(L^2(\GG))^{\GG}$,  which preserves the vector state $h$ implemented by $\Lambda(1)$. If $\GG$ is of Kac type, $E$ maps $L^\infty(\GG)$ to $L^\infty_z(\GG)$, $C(\GG)$ to $C_z(\GG)$ and $\Pol(\GG)$ to $\Pol_z(\GG)$.
\end{prop}
We remark that the Haar state $h$ on $C(\GG)$ is exactly the vector state implemented by $\Lambda(1)$, i.e.~$h(a)=\inn{\Lambda(1), a\Lambda(1)}$,  and it is therefore consistent to also use the symbol $h$ for its extension to $\BB(L^2(\GG))$, as done in Proposition \ref{prop:cond-ext-prop}.

\begin{proof}
The map $E$ is unital and completely positive (henceforth abbreviated ucp) by construction.  Moreover,  $T\mapsto 1\tens T$ is normal as a $*$-isomorphism of $\BB(L^2(\GG))$ onto the von Neumann algebra  $1\tens \BB(L^2(\GG))$ \cite[Section III, Corollary 2.2.12]{Blackadar:OperatorAlgebras}. From this it follows that $\delta$ is normal and since also 
$h \tens \id$ 
is normal, we conclude that $E$ is normal. \\

For $a\in \BB(L^2(\GG))^{\GG} $, we have 
$
E(a)=(h\tens \id)\delta(a)=(h\tens \id )(1\tens a)=a.
$
Conversely, for a given $a\in \BB(L^2(\GG))$, to see that $E(a)\in \BB(L^2(\GG))^{\GG}$ we  pick a net   $M_i=\sum_{k=0}^{n_i} S_i^k\tens T_i^k  \in  L^\infty(\GG))\odot \BB(L^2(\GG))$ converging ultraweakly to $\delta(a)$. 
Since all maps involved are normal (i.e.~ultraweakly continuous) we may now proceed with an approximation argument as follows:
\begin{align*}
\delta(E(a))&=\delta\big((h\tens \id)\delta(a)\big)\\
&=\delta\big((h\tens \id)\lim_iM_i\big)\\
&=\lim_i \delta((h\tens \id)M_i)\\
&=\lim_i \sum_{k=1}^{n_i} h(S_i^k)\delta(T_i^k) \\
&=\lim_i \sum_{k=1}^{n_i} (h\tens \id \tens \id)\big((\id\tens \delta)(S_i^k\tens T_i^k)\big)\\
&=(h\tens \id \tens \id)(\id\tens \delta)(\delta(a))\\
& =(h\tens \id \tens \id)(\Delta\tens \id)(\delta(a))\\
&=((h\tens \id)\Delta \tens \id)(\delta(a))\\
& =(h(-)1\tens \id)(\delta(a))\\
&=\lim_{i} \big(h(-)1\tens \id\big)\Big(\sum_{k=1}^{n_i} S_i^k\tens T_i^k\Big)\\
&= \lim_{i} \sum_{k=1}^{n_i}h (S_i^k)1\tens T_i^k\\
&=1\tens (h\tens \id)\delta(a)\\
&=1\tens E(a).
\end{align*}
 That is, $E$ is a conditional expectation onto the fixed point algebra. Using that $V$ and $W$ fix $\Lambda(1)\tens \Lambda(1)$, we obtain
\begin{align*}
h(E(a))&=h((h\tens \id) \delta(a))\\
&= (h\tens h)\delta(a)\\
&=\inn{ \Lambda(1)\tens \Lambda(1) , \delta(a) (\Lambda(1)\tens \Lambda(1)) }\\
&=\inn{ \Lambda(1)\tens \Lambda(1),  \Sigma V^*\Sigma W^*(1\tens a) W \Sigma V\Sigma   (\Lambda(1)\tens \Lambda(1)) }\\
&=\inn{\Lambda(1)\tens \Lambda(1), (1\tens a)(\Lambda(1)\tens \Lambda(1))}\\
&=h(a).
\end{align*}
It only remains to be shown that $E$ preserves the algebras associated with $\GG$ in the Kac type situation. Arguing as in  \cite[proof of Theorem 1.4]{Lemeux:Haagerup}, we first note that  for $x\in C(\GG)$ we have
\[
\|\Delta(x)\|_2^2:= (h\tens h)(\Delta(x)\Delta(x)^*)=(h \tens h)(\Delta(xx^*))=h(xx^*)=\|x \|_2^2,
\]
and $\Delta$ therefore extends isometrically to $\tilde{\Delta}\colon L^2(\GG) \to L^2(\GG) \hat{\tens}L^2(\GG)$. For $a\in C(\GG)$, we now claim that
\begin{align}\label{eq:cond-exp-formula}
E(a)=\tilde{\Delta}^* \circ \Delta^{\op}(a) \circ\tilde{\Delta},
\end{align}
where $\Delta^{\op}(a):=\sigma(\Delta(a))$. To see this, we  fix $x,y\in \Pol(\GG)$ and compute

\begin{align*}
	\inn{ \Lambda(x),  E(a) \Lambda(y)}&=\inn{ \Lambda(x), (h\tens 1)\delta(a) \Lambda(y)}\\
	&=\inn{\Lambda(1)\tens \Lambda(x), \delta(a) (\Lambda(1)\tens \Lambda(y))}\\
	&=\inn{\Lambda(1)\tens \Lambda(x), \Sigma V^*\Sigma W^*(1\tens a) W\Sigma V\Sigma (\Lambda(1)\tens \Lambda(y))}\\
	&= \inn{V(\Lambda(x)\tens \Lambda(1)), \Sigma\Delta(a) \Sigma V(\Lambda(y)\tens \Lambda (1))}\\
	&=\inn{\Lambda\tens \Lambda(\Delta(x)), \Sigma \Delta(a) \Sigma (\Lambda\tens \Lambda(\Delta(y)))}\\
	&= \inn{\tilde{\Delta}(\Lambda(x)), \Delta^{\op}(a) \circ \tilde{\Delta}(\Lambda(y))}\\
	&=\inn{\Lambda(x), \tilde{\Delta}^*\circ  \Delta^{\op}(a)\circ  \tilde{\Delta} (\Lambda(y))}.
\end{align*}
Using the formula \eqref{eq:cond-exp-formula}, we will now show that when $\GG$ is of Kac type and $a\in \Pol(\GG)$ then $E(a)\in L^\infty(\GG)$. From this it follows that $E$ preserves $L^\infty(\GG)$, since $E$ is  normal and $\Pol(\GG)$ is ultraweakly dense in $L^\infty(\GG)$, and hence that $E$ maps $L^\infty(\GG)$ to $L^\infty_z(\GG)$ by Lemma \ref{lem:fixed-pt-lem}.
Since $h$ is a trace, the map $\Lambda(x)\to \Lambda(x^*)$ extends to an anti-unitary $J$ on $L^2(\GG)$ with the property that $JL^\infty(\GG)J=L^\infty(\GG)'$. By von Neumann's bicommutant theorem and  density of $\Pol(\GG)$ in $L^\infty(\GG)$, it therefore suffices to show that $E(a)$ commutes with $JbJ$ for all $b\in \Pol(\GG)$. To see this, fix again $x,y\in \Pol(\GG)$ and note that $JbJ\Lambda(x)=\Lambda(xb^*)$. Using \eqref{eq:cond-exp-formula}, we now compute as follows:
\begin{align*}
\inn{E(a)JbJ \Lambda(x), \Lambda(y)}&= \inn{\tilde{\Delta}^* \circ \Delta^{\op}(a) \circ\tilde{\Delta} (\Lambda(xb^*)), \Lambda(y)}\\
&=\inn{\Delta^{\op}(a) (\Lambda\tens \Lambda)\Delta(xb^*), (\Lambda\tens \Lambda)\Delta(y)}\\
&=\inn{\Lambda \tens \Lambda(a_{(2)}x_{(1)}b_{(1)}^* \tens a_{(1)}x_{(2)}b_{(2)}^*), \Lambda\tens \Lambda(y_{(1)}\tens y_{(2)})}\\
&=\inn{Jb_{(1)}J\Lambda(a_{(2)}x_{(1)}), \Lambda(y_{(1)})}\cdot \inn{Jb_{(2)}J\Lambda(a_{(1)}x_{(2)}), \Lambda(y_{(2)})}\\
&=\inn{\Lambda(a_{(2)}x_{(1)}), Jb_{(1)}^*J \Lambda(y_{(1)})}\cdot \inn{\Lambda(a_{(1)}x_{(2)}), Jb_{(2)}^*J\Lambda(y_{(2)})}\\
&=\inn{\Lambda(a_{(2)}x_{(1)}),  \Lambda(y_{(1)}b_{(1)})}\cdot \inn{\Lambda(a_{(1)}x_{(2)}), \Lambda(y_{(2)}b_{(2)})}\\
&=\inn{\Lambda\tens \Lambda   (a_{(2)}x_{(1)}\tens a_{(1)}x_{(2)}) ,  \Lambda\tens \Lambda (y_{(1)}b_{(1)}\tens y_{(2)}b_{(2)})}\\
&=\inn{\Delta^{\op}(a) \Lambda\tens \Lambda(\Delta(x)), \tilde{\Delta}\Lambda(yb)}\\
&= \inn{\tilde{\Delta}^* \circ \Delta^{\op}(a) \circ\tilde{\Delta}(\Lambda(x)), Jb^*J\Lambda(y)}\\
&=\inn{JbJ E(a) \Lambda(x), \Lambda(y)}.
\end{align*}
We have thus shown that $E$ maps $L^\infty(\GG)$ to itself, and by Lemma \ref{lem:fixed-pt-lem} it therefore takes values in $L^\infty_z(\GG)$. The restriction of $E$ is thus the unique $h$-preserving conditional expectation from $L^\infty(\GG)$ onto $L^\infty_z(\GG)$. By \cite[Theorem 3.7]{Crann:CaracterDensity} we have $L^\infty_z(\GG)=\Pol_z(\GG)''$ and $E$ is therefore identical to the conditional expectation considered in the proof of \cite[Lemma 6.3]{Wang:lacunary}. In particular, we  have the formula 
\begin{align}\label{eq:E-on-matrix-units}
E(u_{ij}^\alpha)=\delta_{ij} d_\alpha^{-1}\chi(u^\alpha)
\end{align}
 derived in the proof of   \cite[Lemma 6.3]{Wang:lacunary}, and hence that $E$ maps $\Pol(\GG)$ to $\Pol_z(\GG)$ and thus, by continuity, $C(\GG)$ to $C_z(\GG)$. For the convenience of the reader, we now derive the formula \eqref{eq:E-on-matrix-units} directly. To this end, first note that since $\Lambda(1)$ is separating for $L^\infty(\GG)$ it suffices to show that 
\[
\inn{E(u_{ij}^\alpha)\Lambda(1), \Lambda(u_{pq}^\beta)}= \inn{\delta_{ij} d_\alpha^{-1}\chi(u^\alpha)\Lambda(1), \Lambda(u_{pq}^\beta)},
\]
for all $\beta\in \Irred(\GG)$ and $p,q\in \{1,\dots, d_\beta\}$. Note also that since $\GG$ is assumed Kac, the orthogonality relations  \cite[Proposition 5.3.8 (iii)]{Timmermann-book} simplify to the expression
\[
h((u_{ij}^\alpha)^*
u_{pq}^\beta )=\delta_{\alpha,\beta}\delta_{ip}\delta_{jq} d_\alpha^{-1}.
\]
The claimed identity now follows from this and \eqref{eq:cond-exp-formula} via a direct computation:
\begin{align*}
\inn{E(u_{ij}^\alpha)\Lambda(1), \Lambda(u_{pq}^\beta)} &= \Big\langle{\Delta^{\op}(u_{ij}^\alpha)(\Lambda(1)\tens \Lambda(1)), \sum_{l=1}^{d_\beta}\Lambda(u_{pl}^\beta)\tens \Lambda(u_{lq}^\beta) }\Big\rangle\\
&= \Big\langle{ \sum_{k=1}^{d_\alpha}\Lambda(u_{kj}^\alpha)\tens \Lambda(u_{ik}^\alpha) , \sum_{l=1}^{d_\beta}\Lambda(u_{pl}^\beta)\tens \Lambda(u_{lq}^\beta) }\Big\rangle\\
&=\delta_{\alpha,\beta}\delta_{i,j}\delta_{p,q} d_\alpha^{-2}\\
&= \inn{\delta_{ij} d_\alpha^{-1}\chi(u^\alpha)\Lambda(1), \Lambda(u_{pq}^\beta)}. \qedhere
\end{align*}

\end{proof}

\subsection{The conditional expectation is a slip-norm contraction }

Our next aim is to show that the conditional expectation introduced in the previous section is a contraction for the seminorm $L_\ell$ associated with a proper length function $\ell$. Let therefore $\GG$ be a compact quantum group of Kac type, and fix  a proper length function $\ell\colon \Irred(\GG)\to [0,\infty)$. 
On $L^2(\GG) \hat{\tens}L^2(\GG)$ we  consider the densely defined symmetric operator
\begin{align*}
\id \tens \C D_\ell\colon \Lambda(\Pol(\GG)) \odot \Lambda(\Pol (\GG)) &\longrightarrow  L^2(\GG) \hat{\tens}L^2(\GG)
\end{align*}
defined by $\xi\tens \eta \longmapsto \xi \tens \C D_\ell (\eta)$.

\begin{lemma}\label{lem:commutation-lem}
The operator $\id \tens \mathcal{D}_\ell$ commutes with $W,W^*, \Sigma V\Sigma, \Sigma V^*\Sigma, Z$ and $Z^*$, as operators  on the dense subspace $ \Lambda(\Pol(\GG)) \odot \Lambda(\Pol (\GG))\subset L^2(\GG) \hat{\tens}L^2(\GG)$.
\end{lemma}

\begin{proof}
By the  remarks  in Section \ref{sec:corep}, all  maps involved preserve  the dense subspace $ \Lambda(\Pol(\GG)) \odot \Lambda(\Pol (\GG))$ and the claimed commutation relations  are therefore well-defined. 
Computing with matrix coefficients we obtain
\begin{align*}
(\id \tens  \C D_\ell) W^* (\Lambda (u_{ij}^\alpha) \tens \Lambda(u_{pq}^{\beta}))&=(\id \tens  \C D_\ell) (\Lambda \tens \Lambda)(\Delta(u_{pq}^\beta)(u_{ij}^\alpha \tens 1))\\
&=(\id \tens  \C D_\ell) \Big(\sum_{k=1}^{d_{\beta}} \Lambda(u_{pk}^\beta u_{ij}^{\alpha})\tens \Lambda(u_{kq}^\beta)\Big)\\
&= \ell({\beta})\cdot \sum_{k=1}^{d_{\beta}}\Lambda(u_{pk}^\beta u_{ij}^{\alpha})\tens \Lambda(u_{kq}^\beta)\\
&= W^* (\id \tens  \C D_\ell)  \Lambda((u_{ij}^\alpha) \tens \Lambda(u_{pq}^{\beta})).
\end{align*}
From this it follows that $\id \tens  \mathcal{D}_\ell$ commutes with $W^*$. Using that $\id \tens  \C D_\ell$ is symmetric,  we obtain that $\id \tens  \C D_\ell$ also commutes with $W$ as operators on $ \Lambda(\Pol(\GG)) \odot \Lambda(\Pol (\GG))$. 
A computation similar to the one above, shows that $\Sigma V\Sigma $ commutes with $\id \tens  \mathcal{D}_\ell$ and hence the same holds true for its adjoint by the symmetry of $\id \tens  \mathcal{D}_\ell$. Since $Z:= W\Sigma V\Sigma$ it now follows that $\id \tens  \mathcal{D}_\ell$ also commutes with $Z$ and $Z^*$.
\end{proof}




\begin{prop}\label{prop:cond-exp-contraction}
It holds  that $L_{\ell}(E(a))\leq L_{\ell}(a)$ for all $a\in \Pol(\GG)$.
\end{prop}
\begin{proof}
For $a,x,y\in \Pol(\GG)$, we  compute as follows, invoking Lemma \ref{lem:commutation-lem}  in the fourth step:
\begin{align*}
	&\inn{\Lambda(x), \partial_\ell(E(a))\Lambda(y)}=\\
	& = \inn{\Lambda(x), \mathcal{D}_\ell ((h\tens \id  ) \delta(a))\Lambda(y) )}  -  \inn{ \Lambda(x), ((h\tens \id) \delta(a)) \mathcal{D}_\ell \Lambda(y) )}\\
	&=\inn{\Lambda(1) \tens  \mathcal{D_\ell}\Lambda(x) , \delta(a) (\Lambda(1)\tens \Lambda(y))  } - \inn{ \Lambda(1)\tens \Lambda(x) , \delta(a) (\Lambda(1)\tens \mathcal{D}_\ell \Lambda(y))  }\\
	&=\inn{\Lambda(1) \tens  \mathcal{D_\ell}\Lambda(x) , Z^*(1\tens a)Z (\Lambda(1)\tens \Lambda(y))  } - \inn{ \Lambda(1)\tens \Lambda(x) , Z^*(1\tens a)Z (\Lambda(1)\tens \mathcal{D}_\ell \Lambda(y)) }\\
	&=\inn{\Lambda(1) \tens  \Lambda(x) , Z^*(1 \tens  \mathcal{D_\ell}a)Z (\Lambda(1)\tens \Lambda(y)) } - \inn{ \Lambda(1)\tens \Lambda(x) , Z^*(1\tens a\mathcal{D}_\ell)Z (\Lambda(1)\tens  \Lambda(y))  }\\
	&=\inn{ \Lambda(1)\tens \Lambda(x)  , Z^* (1\tens \partial_\ell(a)Z (\Lambda(1)\tens \Lambda(y))}\\
	&=\inn{ \Lambda(x) , (h\tens \id)\delta(\partial_\ell(a)) \Lambda(y) }\\
	&=\inn{  \Lambda(x) , E(\partial_\ell(a))\Lambda(y)}.
\end{align*}
We thus obtain that $\partial_\ell(E(a))=E(\partial_\ell(a))$ and since $E$ is a norm contraction the desired conclusion follows:
\[
L_\ell(E(a))=\|\partial_\ell(E(a))\|=\|E(\partial_\ell(a))\|\leq \|\partial_\ell(a)\|=L_\ell(a). \qedhere
\]
\end{proof}
The last bit of information about $L_\ell$ needed for the proof of Theorem \ref{introthm:lifting}, is that it satisfies Li's left-invariance condition \cite{Li:ECQ}; i.e.~that the following result holds true.

  \begin{lemma}\label{lem:left-inv}
 The seminorm $L_\ell$ satisfies  $L_\ell\big((\varphi\tens \id)\Delta(a)\big)\leq \|\varphi\|\cdot L_\ell(a)$ for all $a\in \Pol(\GG)$ and $\varphi\in \BB(L^2(\GG))^*$.
 \end{lemma}

 \begin{proof}
For $a\in \Pol(\GG)$, we  compute as follows, where all equalities should  be interpreted as holding on the dense subspace $ \Lambda(\Pol(\GG)) \odot \Lambda(\Pol (\GG))\subset L^2(\GG) \hat{\tens}L^2(\GG) $:

\begin{align*}
[\id \tens  \C D_{\ell}, \Delta(a)] &=(\id \tens  \C D_\ell ) \Delta(a)-\Delta(a)(\id \tens  \mathcal{D}_\ell) \\
 &= (\id \tens  \C D_\ell)W^*(1\tens a)W- W^*(1\tens a)W (\id \tens  \C D_\ell)\\
 &=W^*(\id \tens  \C D_\ell)(1\tens a)W- W^*(1\tens a) (\id \tens  \C D_\ell)W \tag{Lemma \ref{lem:commutation-lem}}\\
 &= W^* (1 \tens [\C D_\ell,a])W.
 \end{align*}
 We therefore obtain the following identities, between  operators defined on the dense subspace $\Pol(\GG)\subset L^2(\GG)$:
  \begin{align*}
[\C D_\ell, (\varphi\tens \id)\Delta(a)]&=\varphi(a_{(1)})[\C D_\ell, a_{(2)}] \\
&=(\varphi\tens\id) [\id \tens  \C D_\ell, \Delta(a) ])\\
&= (\varphi\tens\id) (W^* (1 \tens [\C D_\ell,a])W).
 \end{align*}
 The bounded extension of $[\C D_\ell, (\varphi\tens \id)\Delta(a)]$, i.e.~the operator $\partial_\ell( (\varphi\tens \id)\Delta(a))$, therefore agrees with $(\varphi\tens\id) (W^* (1 \tens \partial_\ell(a))W)$, from which the claim now follows via the following estimates:
\begin{align*}
 L_\ell((\varphi\tens \id )\Delta(a)) &=\| \partial_\ell((\varphi\tens \id )\Delta(a))\|\\
 &= \| (\varphi\tens\id) (W^* (1 \tens \partial_\ell(a))W) \|\\
 &\leq  \|\varphi \tens \id\| \cdot \| (W^* (1 \tens \partial_\ell(a))W) \|\\
 &=  \|\varphi \| \cdot\|  \partial_\ell(a) \|\\
 &=\|\varphi\| \cdot L_\ell(a).  \qedhere
\end{align*}

 \end{proof}

\section{Central approximations of the counit}
In this section, we show how the counit on a compact coamenable quantum group $\GG$ can be approximated by \emph{finitely supported  central states}. To this end, recall  that a state $\varphi\in \mathcal{S}(C(\GG))$ is called \emph{central} if
\[
(\varphi \tens \id)\Delta=(\id\tens \varphi)\Delta.
\]
So, the Haar state is always central and so is the  counit, which extends to a state since  $\GG$ is assumed coamenable. For our purposes, the following characterisation of centrality (which might be well-known to experts in the field) will be the most convenient to work with.

\begin{lemma}\label{lem:centrality-in-terms-of-E}
When $\GG$ is of Kac type then $\varphi\in \S(C(\GG))$ is central if and only if $\varphi=\varphi\circ E$, where {$E\colon C(\GG) \to C_z(\GG)$ is the conditional expectation} from Proposition \ref{prop:cond-ext-prop}.
\end{lemma}
\begin{proof}
We first assume that $\varphi$ is central and compute with matrix coefficients as follows:
\begin{align}\label{eq:centrality-eq}
\sum_{k=1}^{d_\alpha} \varphi(u_{ik}^\alpha)u_{kj}^\alpha=(\varphi\tens \id)\Delta(u_{ij}^\alpha)=
(\id \tens \varphi)\Delta(u_{ij}^\alpha)=\sum_{k=1}^{d_\alpha} \varphi(u_{kj}^\alpha)u_{ik}^\alpha
\end{align}
But since $\{u_{ij}^\alpha \mid \alpha \in \Irred(\GG), 1\leq i,j \leq d_\alpha\}$ is a linear basis for $\Pol(\GG)$, this forces $\varphi(u_{ik}^\alpha)=0$ if $k\neq i$ and $\varphi(u_{kj}^\alpha)=0$ if $k\neq j$. Hence, equation \eqref{eq:centrality-eq} reduces to $\varphi(u_{ii}^\alpha)u_{ij}^\alpha=\varphi(u_{jj}^\alpha)u_{ij}^\alpha$ and we conclude that $\varphi(u_{ii}^\alpha)=\varphi(u_{jj}^\alpha)$. In other words, on the matrix $(u^\alpha_{ij})_{i,j=1}^{d_\alpha}$ the central state $\varphi$ vanishes on off-diagonal elements and is constant down the diagonal.  Using the formula \eqref{eq:E-on-matrix-units},  we therefore obtain
\begin{align*}
\varphi\circ E(u_{ij}^\alpha)=\varphi\big(\delta_{i,j}d_\alpha^{-1}\chi(u^\alpha)\big)=\delta_{i,j}d_{\alpha}^{-1}\sum_{k=1}^{d_\alpha}\varphi(u_{kk}^\alpha)=\delta_{i,j} \varphi(u_{ii}^\alpha)=\varphi(u_{ij}^\alpha).
\end{align*}
By linearity and continuity, it now follows that $\varphi\circ E=\varphi$.\\

For the converse, assume now that $\varphi\circ E=\varphi$. Computing with matrix coefficients  (using equation \eqref{eq:E-on-matrix-units} in the second step) now gives
\begin{align*}
(\varphi\tens \id)(\Delta(u_{ij}^\alpha))= \sum_{k=1}^{d_\alpha} \varphi(E(u_{ik}^\alpha))u_{kj}^\alpha
= \sum_{k=1}^{d_\alpha} \sum_{l=1}^{d_\alpha} \varphi(\delta_{i,k}d_\alpha^{-1} (u_{ll}^\alpha))u_{kj}^\alpha
= d_\alpha^{-1} \sum_{l=1}^{d_\alpha} \varphi( u_{ll}^\alpha)u_{ij}^\alpha,
\end{align*}
while slicing on the other leg gives
\begin{align*}
(\id \tens \varphi)(\Delta(u_{ij}^\alpha))= \sum_{k=1}^{d_\alpha} \varphi(E(u_{kj}^\alpha))u_{ik}^\alpha
= \sum_{k=1}^{d_\alpha} \sum_{l=1}^{d_\alpha} \varphi(\delta_{k,j} d_\alpha^{-1} \varphi(u_{ll}^\alpha) )u_{ik}^\alpha
=  d_\alpha^{-1}\sum_{l=1}^{d_\alpha} \varphi( u_{ll}^\alpha )u_{ij}^\alpha.\\
\end{align*}
By linearity, density and continuity we therefore conclude that $\varphi$ is central.
\end{proof}

 The approximation property alluded to in the introduction to the present section now takes the following form:

 \begin{lemma}\label{lem:foelner-states}
 If $\GG$ is a compact, coamenable quantum group of Kac type, then there exists a sequence of  states $(\chi_n)_{n\in \NN}$ on $C(\GG)$ such that
 \begin{enumerate}
 \item Each $\chi_n$ is central and supported  in only finitely many matrix coefficients.
\item The sequence $\chi_n$ converges to  $\epsilon$ in the weak$^*$ topology on $\S(C(\GG))$.
 \end{enumerate}
 \end{lemma}
 
 This result can be derived in numerous ways from the existing characterisations of coamenability (see e.g.~ \cite{tomatsu-amenable}). Here the result will be derived using Følner sequences, since this approach makes the connection to the corresponding result for discrete groups transparent.

 \begin{proof}
 Denote by $F_n\subset \Irred(\GG)$ a Følner sequence for $\GG$  \cite{DK:L2-of-coamenable}, and define a sequence of functions $\omega_n\colon \Irred(\GG) \to[0,\infty)$ by
 \[
 \omega_n(\gamma):=\sum_{\alpha,\beta\in F_n} \frac{N_{\alpha,\bar{\beta}}^\gamma d_\alpha d_\beta }{d_\gamma\big(\sum_{\xi\in F_n}d_\xi^2 \big)}
 \]
Then \cite[Lemma 5.5]{HWW:Pointwise-convergence} 
 shows that $\lim_{n}\omega_n(\gamma)=1$ for all $\gamma \in \Irred(\GG)$ and that the associated multipliers $T_n\colon \Pol(\GG) \to \Pol(\GG)$ 
 \[
 T_n(u_{ij}^\alpha):=\omega_n(\alpha) u_{ij}^\alpha
 \]
 extend to ucp maps on $L^\infty(\GG)$, and hence also on $C(\GG)$ since $T_n$ preserves $\Pol(\GG)$. Setting $\chi_n:=\epsilon \circ T_n\colon C(\GG) \to \CC $ we therefore obtain a sequence of states satisfying
 \[
 \chi_n(u_{ij}^\alpha)=\epsilon(\omega_n(\alpha) u_{ij}^\alpha)=\omega_n(\alpha)\delta_{ij}.
 \]
 This formula immediately implies 
 that $\chi_n$ is central  and moreover that
 \[
  \chi_n(u_{ij}^\alpha)=\omega_n(\alpha)\delta_{ij} \underset{n\to \infty}{\longrightarrow} \delta_{ij}=\epsilon (u_{ij}^\alpha).
 \]
 By linearity, the sequence of states $\chi_n$ therefore converges to $\epsilon$ pointwise on $\Pol(\GG)$ and by density also on $C(\GG)$. Lastly, to see that $\chi_n$ is only supported in finitely many matrix coefficients, note that $\chi_n(u_{ij}^\gamma)\neq 0$ if and only if $N_{\alpha, \bar{\beta}}^\gamma \neq 0 $ for some $\alpha, \beta\in F_n$. But since $F_n$ is finite, the union of supports of products of the form $\alpha \cdot \bar{\beta}$ with $\alpha,\beta\in F_n$ is also finite,  so when $\gamma$ falls outside this finite set we have $\chi_n(u_{ij}^\gamma)=0$.
 \end{proof}

{
The main reason why the central approximation $\chi_n$ of the counit provided by Lemma \ref{lem:foelner-states} is relevant to us, is contained in the following lemma, which shows that the distance between $\chi_n$ and counit can be measured by their restrictions to the algebra of central functions. }

 \begin{lemma}\label{lem:restriction}
Let $\GG$ be a compact, coamenable quantum group of Kac type  and let $\ell\colon \Irred(\GG) \to [0,\infty)$ be a proper length function.  
For any two central states, $\varphi$ and $\psi$, on $C(\GG)$ it holds that 
 \begin{align*}
 d_\ell(\varphi,\psi)&:=\sup\{|\varphi(a)-\psi(a)|  \mid a\in \Pol(\GG), L_l(a)\leq 1 \} \\
 &= \sup\{|\varphi(a)-\psi(a)| \mid a\in \Pol_z(\GG), L_l(a)\leq 1 \} \\
 &=:d_\ell^z\big(\varphi\hspace{-0.15cm} \restriction_{C_z(\GG)}, \psi\hspace{-0.15cm} \restriction_{C_z(\GG)}\big)
 \end{align*}
 
\begin{proof}
The inequality ``$\geq $'' is trivial. For the opposite, let $a\in \Pol(\GG)$ with $L_\ell(a)\leq 1 $ be given. Then $E(a)\in \Pol_z(\GG)$ by Proposition \ref{prop:cond-ext-prop} and $L_\ell(E(a))\leq L_\ell(a)\leq 1$ by Proposition \ref{prop:cond-exp-contraction}. Moreover, since $\varphi$ and $\psi$ are assumed central, Lemma \ref{lem:centrality-in-terms-of-E} shows that $|\varphi(E(a))-\psi(E(a))|=|\varphi(a)-\psi(a)|$ which implies the desired inequality. 
 \end{proof}
 \end{lemma}

 We are now ready to prove the essential estimate needed for the proof of Theorem \ref{introthm:lifting}. 
 For the statement, we first introduce a bit of notation.
 \begin{dfn}\label{def:berezin}
For a compact quantum group $\GG$ and $\chi\in \S(C(\GG))$ we denote by $\beta_\chi\colon C(\GG) \to C(\GG)$ the ucp map defined by $\beta_\chi(a)=(\id \tens \chi)\Delta(a)$.
 \end{dfn}

 \begin{lemma}\label{lem:berezin-approximation}
Let $\GG$ be a compact, coamenable quantum group and let $\ell\colon \Irred(\GG) \to [0,\infty)$ be a proper length function. For any $\chi\in \S(C(\GG))$ and any $a\in C(\GG)$ it holds that $\|a-\beta_\chi(a)\|\leq d_\ell(\epsilon, \chi)\cdot L_\ell(a)$.
 \end{lemma}
 We remark that the estimate in Lemma \ref{lem:berezin-approximation}, as well as its proof, is heavily inspired by the analysis of the quantum metric structure of $q$-deformations, studied in \cite{KaadKyed2022, AKK:Podcon, AKK:PolyApprox}.
 
 \begin{proof}
 Take any $\varphi\in \BB(L^2(\GG))^*$ with $\|\varphi\|=1$. Then 
 \begin{align*}
 |\varphi (a-\beta_\chi(a))|&=|\varphi((\id\tens (\epsilon-\chi))\Delta(a) )|=|(\epsilon-\chi)((\varphi\tens \id)(\Delta(a)))|\\
 &\leq d_{\ell}(\epsilon, \chi)\cdot L_\ell((\varphi\tens \id)(\Delta(a))) \leq  d_{\ell}(\epsilon, \chi)\cdot L_\ell(a),
 \end{align*} 
 where the last inequality follows from Lemma \ref{lem:left-inv}. Since 
 \[
 \| a-\beta_\chi(a)\|=\sup\{ |\varphi(a-\beta_\chi(a))| \mid \varphi\in \BB(L^2(\GG))^*, \|\varphi\|\leq 1 \},
 \]
 the desired estimate now follows.
 \end{proof}
 
 With the above results at our disposal, we are now ready to proceed with the proofs of Theorem \ref{introthm:lifting} and Corollary \ref{cor:equivalent-fusion} from the introduction. 
 
 \begin{proof}[Proof of Theorem \ref{introthm:lifting}]
 We aim to apply Kaad's characterisation of compact quantum metric spaces presented in Theorem \ref{thm:kaad-criterion}.  To meet the criteria set forth in Theorem \ref{thm:kaad-criterion}, it suffices to show that $(C(\GG), L_\ell)$ has finite diameter (in the sense of Definition \ref{def:finite-diameter}) and provide a sequence of finite rank ucp maps $\beta_n\colon C(\GG) \to C(\GG)$, which converges to the identity uniformly on the $L_\ell$-unit ball.   To this end, consider again the sequence of central states, $(\chi_n)_{n\in \NN}$  provided by Lemma \ref{lem:foelner-states}. We  augment the sequence by setting $\chi_0:=h$ and put $\beta_n:=\beta_{\chi_n}$; see Definition \ref{def:berezin}. 
Using the centrality of $\chi_n$ and $\epsilon$ together with Lemma \ref{lem:berezin-approximation} and Lemma \ref{lem:restriction}  we the obtain:
 \begin{align}\label{eq:berezin-approx}
\|a-\beta_n(a)\|\leq d_\ell(\epsilon, \chi_n)\cdot L_\ell(a)= d_\ell^z( 
  \epsilon\hspace{-0.15cm} \restriction_{C_z(\GG)}, \chi_n\hspace{-0.15cm} \restriction_{C_z(\GG)})\cdot L_\ell(a),
 \end{align}
 for all $a\in \Pol(\GG)$.
Since $d_\ell^z$ metrizes the weak* topology on $\S(C_z(\GG))$ by assumption, we have that 
\[
d_\ell^z( 
  \epsilon\hspace{-0.15cm} \restriction_{C_z(\GG)}, \chi_n\hspace{-0.15cm} \restriction_{C_z(\GG)})<\infty  \T{ for all $n\in \NN_0$ and } \lim_{n\to \infty}d_\ell^z( 
  \epsilon\hspace{-0.15cm} \restriction_{C_z(\GG)}, \chi_n\hspace{-0.15cm} \restriction_{C_z(\GG)})=0.
\] 
For $n=0$ we have $\chi_0=h$, so the invariance property of $h$  and \eqref{eq:berezin-approx} therefore yield
\[
{\|a- h(a)1\|=\| a-\beta_0(a)\|\leq d_\ell^z( 
  \epsilon\hspace{-0.15cm} \restriction_{C_z(\GG)}, h\hspace{-0.15cm} \restriction_{C_z(\GG)}))\cdot L_\ell(a)},
\]
and hence  $(C(\GG),L_\ell)$ has finite diameter. Moreover,  since $\beta_n$ is ucp with finite-dimensional image and $\lim_{n\to \infty}d_\ell^z( 
  \epsilon\hspace{-0.15cm} \restriction_{C_z(\GG)}, \chi_n\hspace{-0.15cm} \restriction_{C_z(\GG)})=0$, the conditions in \cite[Theorem 3.1]{Kaad2023} are fulfilled and the proof is complete.
\end{proof}

\begin{proof}[Proof of Corollary \ref{cor:equivalent-fusion}]
The assumed bijection $\alpha$ extends to a unitary $ \ell^2(\Irred(\GG_1)) \simeq \ell^2(\Irred(\GG_2))$ which intertwines the two fusion algebras $F(\GG_1)$ and $F(\GG_2)$ as well as the Dirac operators $\tilde{D}_{\ell_1}$ and $\tilde{D}_{\ell_2}$. Hence, if $(F(\GG_1), \ell^2(\Irred(\GG_1)), \tilde{D}_{\ell_1})$ is a spectral metric space, then so is $(F(\GG_2), \ell^2(\Irred(\GG_2)), \tilde{D}_{\ell_2})$. By Corollary \ref{cor:fusion-vs-central}, this means that $(C_z(\GG_2), L_{\ell_2}$) is a compact quantum metric space and  since $\GG_2$ is assumed coamenable and of Kac type,  Theorem \ref{introthm:lifting}  gives that $(C(\GG_2), L_{\ell_2})$ is a compact quantum metric space as well.
\end{proof}

\begin{remark}
Throughout the present paper, we have fixed the domain of the slip-norm $L_\ell$ to be $\Pol(\GG)$, which can be considered  the minimal natural choice of domain. Our main results also make sense for the maximal choice of domain 
\[
C_{\text{Lip}}(\GG):=\{a\in C(\GG) \mid a(\text{Dom}(D_\ell)) \subseteq \text{Dom}(D_\ell) \ \text{ and } \ [D_\ell, a] \text{ extends boundedly} \}
\]
on which the formula defining $L_\ell$  yields a finite number.  It is possible that the techniques used to prove Theorem \ref{introthm:lifting} and Corollary \ref{cor:equivalent-fusion} can be modified to show the corresponding (and stronger) results for the maximal version of $L_\ell$, but to minimise the technicalities stemming from unbounded operator theory, we shall not pursue this question further.

\end{remark}

 \section{Examples}\label{sec:examples}
 The aim of the present section is to prove Corollary \ref{introcor:examples}, and thereby provide concrete examples of length functions on quantum groups whose associated seminorms give compact quantum metric structures.  Our primary focus will therefore be on compact matrix quantum groups \cite[Section 6.1]{Timmermann-book}, and below we will restrict further to the situation where the fundamental corepresentation is equivalent to its conjugate. In this situation, every irreducible corepresentation $\alpha$ is equivalent to a subrepresentation of some iterated tensor power of the fundamental unitary corepresentation, and setting $\ell(\alpha)$ equal to the first tensor power containing $\alpha$ defines a length function on the quantum group in question. 
  We first prove the following result regarding length functions of this type.

 \begin{lemma}\label{lma:OR-coreps}
 	Let $\GG$ be a  compact matrix quantum group with fundamental unitary corepresentation $v$. Suppose, moreover, that $v$  is equivalent to its conjugate and that the associated length function $\ell\colon \Irred(\GG) \to \nn_0$ is injective. If the set of structure constants $\{N_{\alpha,\beta}^{\gamma}\mid \alpha,\beta,\gamma\in \Irred (\GG) \}$ is bounded, then $(F(\GG), \ell^2(\Irred (\GG), \tilde{D}_{\ell}) $ is a spectral metric space.
 \end{lemma}

 
 \begin{proof}
 Denote the image of $\ell$ as $\{n_0,n_1,n_2,\ldots\}\subseteq \NN_0$ (listed in increasing order) and label the elements in  $\Irred(\GG)$ accordingly as $\{\alpha_0, \alpha_1,\alpha_2, \ldots\}$ so that $\ell(\alpha_k)=n_k$ for all $k\in \NN_0$.   For $n\in \NN_0$, denote by $A_n$ the finite-dimensional subspace $\mathrm{span}_\CC \{ \alpha\in \Irred(\GG) \mid \ell(\alpha) \leq n \} \subseteq F(\GG)$. In Section \ref{sec:spec-trip-fusion}, we saw  that the fusion algebra together with its natural trace and the filtration $(A_n)_{n\in \nn_0}$
 falls within the class studied by Ozawa and Rieffel in \cite{OzawaRieffel2005}, and we now wish to use their Main Theorem 1.2 to show our result. Denote by $P_n \in \BB (\ell^2 (\Irred(\GG)))$ the projection onto the subspace spanned by irreducible corepresentations of length exactly $n$, with the convention that $P_n=0$ if there are no such corepresentations.    Decomposing an element $a \in F(\GG)$, accordingly, as a finite sum $\sum_{k=0}^\infty a_k$ as in \cite[Main Theorem 1.2]{OzawaRieffel2005}, we have $a_k=0$ if $k\notin \{n_0,n_1, n_2, \dots\}$ and for $k\in  \{n_0,n_1, n_2, \dots\}$, $a_k$ takes the form $z_{k}\alpha_k$ for some $z_k\in \CC.$  
  To verify the conditions of \cite[Main Theorem 1.2]{OzawaRieffel2005}, we must  provide a constant $C > 0$ such that for all $a \in F(\GG)$ and all $k,m,n \in \NN_0$ it holds that
 	\begin{align}\label{eq:haagerup-type-condition}
 		\Vert P_m a_k P_n \Vert \leq C \Vert a_k \Vert_2.
 	\end{align}
The only non-trivial case is when $k\in  \{n_0,n_1, n_2, \dots\}$ and  $a_k=z_k \alpha_k$ for some $z_k\in \CC\setminus \{0\}$, and for the purpose of proving an inequality of the form \eqref{eq:haagerup-type-condition} we may as well assume $z_k=1$. Picking a unit vector $\xi\in \ell^2(\Irred(\GG))$, we again have that $P_n\xi=0$ if $n\notin \{n_0, n_1, \ldots\}$ and $P_n\xi=\xi_n\alpha_n$ for some $\xi_n\in \CC$ when 
  $n\in \{n_0, n_1, \ldots\}$. In the latter case, we decompose the product, according to the fusion rules,  as
  \[
  \alpha_k\cdot \alpha_n = \sum_{i=0}^\infty N_{\alpha_k, \alpha_n}^{\alpha_i} \alpha_i
  \]
  to obtain
  \[
  P_m a_k P_n \xi= P_m \alpha_k (\xi_n\alpha_n)= P_m \sum_{i=0}^\infty \xi_nN_{\alpha_k, \alpha_n}^{\alpha_i} \alpha_i =\xi_n N_{\alpha_k, \alpha_n}^{\alpha_m} \alpha_m. 
  \]
Denoting by $M$ the assumed uniform upper bound on the structure constants, we therefore have
\[
\|P_m a_k P_n \xi\|_2=\| \xi_n N_{\alpha_k, \alpha_n}^{\alpha_m} \alpha_m\|_2 \leq |\xi_n|\cdot M \cdot \|\alpha_m\|_2\leq M,
\]
which shows that $C=M$ does the job.
 \end{proof}

With the above lemma at our disposal,  Corollary \ref{introcor:examples} now follows easily. 
 
\begin{proof}[Proof of Corollary \ref{introcor:examples}]
First recall that the irreducible (co-)representations of $SU(2)$  
can be labeled as $\{u^n\}_{n \in \NN_0}$, where $u^n$ is the unique irreducible representation in dimension $n+1$.  The representation $u^1$ is therefore the fundamental representation and equivalent to its conjugate, and we therefore obtain a length function $\ell$ as described in the paragraph preceding Corollary \ref{introcor:examples}. Moreover, the fusion rules are given as
 	\begin{align}\label{eq:SU2-fusion}
 		u^k \cdot u^n = u^{|k-n|} + u^{|k-n|+2} + \cdots + u^{k+n},
 	\end{align}
from which it follows that $\ell(u^n)=n$ and  that the structure constants are uniformly bounded by 1. The assumptions in Lemma \ref{lma:OR-coreps}  are therefore satisfied and we conclude that $(F(SU(2)), \ell^2(\Irred(SU(2))), \tilde{D}_\ell)$ is a spectral metric space. Since $SU(2)$ 
is coamenable and of Kac type, an application of Corollary \ref{cor:fusion-vs-central} and Theorem \ref{introthm:lifting} shows that $(C(SU(2)), L_\ell)$ is a compact quantum metric space. \\

Turning to $O_2^+$, by \cite{Banica-On+-rep} its irreducible corepresentations may also be labeled  as $\{u^n\}_{n \in \NN_0}$ with $u^0$ being the trivial corepresentation,  $u^1$ being   the self-conjugate fundamental unitary corepresentation and with fusion rules described by \eqref{eq:SU2-fusion}. The associated length function therefore also satisfies $\ell(u^n)=n$. Since  $O_2^+$ is coamenable and of Kac type and since we have already proved that $(C(SU(2)), L_\ell)$ is a compact quantum metric space, an application of Corollary \ref{cor:equivalent-fusion} yields the corresponding result for $O_2^+$.\\

It now remains to argue that $C(SO(3))$ and $C(S_4^+)$ acquire compact quantum metric structures from their natural length functions, which will follow by arguments very similar to those appearing in the first part of the proof.  Starting with  $SO(3)$, its irreducible corepresentations can also be labeled  $\{ u^n\}_{n \in \NN_0}$, where $u^0$ is the trivial corepresentation and $u^1$ is the fundamental corepresentation, which is again equivalent to its conjugate.  The fusion rules are given by
 	\begin{align*}
 		u^k \cdot u^m= u^{|k-m|} +u^{|k-m|+1} + \cdots +u^{k+m}
 	\end{align*}
and it therefore follows that $\ell(u^n)=n$ and that the structure constants are uniformly bounded by $1$. By Lemma \ref{lma:OR-coreps}, we therefore conclude that $(F(SO(3)), \ell^2(\Irred(SO(3))), \tilde{L}_\ell)$ is a spectral metric space, and combining Corollary \ref{cor:fusion-vs-central} and Theorem \ref{introthm:lifting} we conclude that $(C(SO(3)), L_\ell)$ is a compact quantum metric space.  To finish the proof, we note that it was shown in \cite{banica-sym} that the corepresentation theory of $S_4^+$ and $SO(3)$ are identical (in the same way as was the case for $SU(2)$ and $O_2^+$) and since $S_4^+$ is both coamenable and of Kac type, 
Corollary \ref{cor:equivalent-fusion} implies that also $(C(S_4^+),L_\ell)$ is a compact quantum metric space.	
	%
\end{proof}
\begin{remark}
The way the proof of   Corollary \ref{introcor:examples} is written, the relationship between the different quantum groups is not that transparent. However, since $SU(2)$ is a double cover of $SO(3)$, the representation theory of the latter is contained in that of the former, and it is therefore no surprise that the result for $SO(3)$ follows once we have it for $SU(2)$. As already explained in the introduction, $O_2^+\simeq SU_{-1}(2)$
and $S_4^+\simeq SO_{-1}(3)$,  and since the corepresentation theory is stable under $q$-deformations, an application of our main results therefore yields the compact quantum metric structures for $O_2^+$ and $S_4^+$ as well. 
\end{remark}

\begin{remark}
We end with an aside regarding  finite-dimensional truncations, a subject which has recently attracted a lot of attention \cite{walter-estrada-cp-grps, walter:GH-convergence, walter-connes:truncations, CvS:Truncations, walter-tori, Leimbach:PWT}). Rieffel's recent paper \cite{Rie:truncations} (see also \cite{Leimbach:PWT}) provides a systematic framework for coactions of quantum groups, within which such finite-dimensional approximations can be obtained. The examples studied above are compatible with the main result (Theorem 6.1)  in \cite{Rie:truncations}; indeed, our context is also restricted to the realm of coamenable compact quantum groups, and in  Lemma \ref{lem:left-inv} we have proven that slip-norms arising from length functions are left invariant. Note, in this respect, that the conventions in \cite{Rie:truncations} are different than those in the present paper, in that the \emph{left} coactions of \cite{Rie:truncations}  are what would be called \emph{right} coactions in our setup.  Paraphrasing  \cite[Theorem 6.1]{Rie:truncations} in the context of (our) left coactions, means that the necessary invariance property of the Lip-norm in question is exactly the one covered by Lemma \ref{lem:left-inv}. Hence, for ergodic  coactions of  the examples covered by  Corollary \ref{introcor:examples},  \cite{Rie:truncations} provides Lip-norms on the algebra acted upon (this construction is originally due to Li \cite{Li:ECQ}) together with a natural sequence of finite-dimensional subspaces converging to the total space in the quantum Gromov-Hausdorff distance \cite{Rie:GHD}.

\end{remark}

\bibliographystyle{alpha}
\bibliography{bibliography-CQMS-from-length.bib}

\end{document}